\documentclass[11pt, twoside]{article}

\usepackage{amsmath}
\usepackage{amsthm}
\usepackage{amssymb}
\usepackage{mathrsfs}  
\usepackage[all]{xy}
\usepackage{multirow}
\usepackage{caption}
\usepackage{float}
\usepackage{multicol}
\usepackage{soul}
\usepackage{xcolor}
\usepackage{enumerate}

\usepackage{setspace}
\usepackage{geometry}
\newcommand{\bF}{{\bf F}}
\newcommand{\matV}{\ensuremath{\mathcal{V}}}
\newcommand{\matGM}{\ensuremath{\mathcal{RM}}}

\newcommand{\val}{\textsf{Val}}

\newcommand{\lfis}{{\bf LFI}s}
\newcommand{\lfi}{{\bf LFI}}

\newcommand{\mbc}{{\bf mbC}}

\newcommand{\cpl}{\text{\bf CPL}}

\newcommand{\mbC}{{\bf mbC}}

\newcommand{\cl}{{\bf cl}}

\newcommand{\mbCcl}{{\bf mbCcl}}

\newcommand{\mbcci}{{\bf mbCci}}
\newcommand{\mbccl}{{\bf mbCcl}}

\newcommand{\cila}{{\bf Cila}}

\newcommand{\CILA}{{\bf Cila}}

\newcommand{\A}{\ensuremath{\mathcal{A}}}

\newcommand{\matF}{\ensuremath{\mathcal{F}}}

\newcommand{\axci}{{\bf ci}}

\newcommand{\axca}{{\bf ca}}
\newcommand{\MP}{\textbf{MP}}

\newcommand{\kax}{{\bf Ax1}}
\newcommand{\axTrans}{{\bf Ax2}}
\newcommand{\axed}{{\bf Ax3}}
\newcommand{\axeea}{{\bf Ax4}}
\newcommand{\axeeb}{{\bf Ax5}}
\newcommand{\axouda}{{\bf Ax6}}
\newcommand{\axoudb}{{\bf Ax7}}
\newcommand{\axoue}{{\bf Ax8}}
\newcommand{\axtnd}{{\bf Ax9}}
\newcommand{\axexp}{{\bf bc}}

\newcommand{\axcf}{{\bf Ax10}}
\newcommand{\axdc}{{\bf dc}}
\newcommand{\axp}{{\bf P}}

\newcommand{\imp}{\to}

\newcommand{\cons}{\ensuremath{{\circ}}}

\newcommand{\sneg}{\ensuremath{{\sim}}}

\newcommand{\laT}{\textsf{T}}
\newcommand{\lat}{\textsf{t}}

\newcommand{\laF}{\textsf{F}}
\newcommand{\laL}{\textsf{L}}
\newcommand{\calT}{\mathcal{T}}
\newcommand{\bbT}{\mathbb{T}}

\newtheorem{theorem}{Theorem}[section]
\newtheorem{lemma}[theorem]{Lemma}
\newtheorem{prop}[theorem]{Proposition}
\newtheorem{coro}[theorem]{Corollary}
\newtheorem{example}[theorem]{Example}
\newtheorem{defi}[theorem]{Definition}
\newtheorem{remark}[theorem]{Remark}

\usepackage{authblk}

\title{\textbf{A simple decision procedure for da Costa's $C_n$\\ logics by Restricted Nmatrix semantics}}
\author{Coniglio, Marcelo E.\thanks{coniglio@unicamp.br} }

\author{Toledo, Guilherme V.\thanks{guivtoledo@gmail.com}}

\affil{Institute of Philosophy and the Humanities - IFCH and\\
Centre for Logic, Epistemology and The History of Science - CLE\\
University of Campinas - Unicamp\\
Campinas, SP, Brazil}

\usepackage{fancyhdr}
\providecommand{\keywords}[1]{\textbf{\textit{Keywords:}} #1}

\begin{document}

\setcounter{page}{1}     

\maketitle

{\em Dedicated to Newton C.A. da Costa, a permanent source of inspiration\\}

\begin{abstract}
Despite being fairly powerful, finite non-deterministic matrices are unable to characterize some logics of formal inconsistency, such as those found between $\mbCcl$ and $\CILA$. In order to overcome this limitation, we propose here restricted non-deterministic matrices (in short, RNmatrices), which are non-deterministic algebras together with a subset of the set of valuations. This allows us to characterize not only \mbccl\ and \cila\ (which is equivalent, up to language, to da Costa's logic $C_1$) but the whole hierarchy of da Costa's calculi $C_n$. This produces  a novel decision procedure for these logics. Moreover, we show that the RNmatrix semantics proposed here induces naturally a labelled tableau system for each $C_n$, which constitutes another decision procedure for these logics. This new semantics allows us to conceive da Costa's hierarchy of $C$-systems as  a family of (non deterministically) $(n+2)$-valued logics, where $n$ is the number of ``inconsistently true'' truth-values and 2 is the number of ``classical'' or ``consistent'' truth-values, for every $C_n$. 
\end{abstract}

\keywords{da Costa's C-systems; paraconsistent logics; non-deterministic semantics; non-deterministic matrices; swap structures; multialgebras; decidability; Dugundji's theorem.}

\section{Introduction}

In 1963 Newton C. A. da Costa presented his {\em Tese de Cátedra} (similar to {\em Habilitation} Thesis) to the Federal University of Paran\'a,\footnote{Called ``Universidade do Paran\'a'' (University of Paran\'a) at that time.} Brazil, entitled ``Sistemas Formais Inconsistentes'' (Inconsistent Formal Systems, see~\cite{dC63}). The thesis, defended in 1964,  is a groundbreaking
work in the field of paraconsistency, that is, in the study of logical systems containing a negation in which not every contradiction (w.r.t. such negation) trivializes. Indeed, his hierarchy $C_n$ (for $n \geq 1$) of {\em $C$-systems} constitutes the first systematic study in the field of paraconsistency.\footnote{The first paraconsistent formal system introduced in the literature is the {\em Discussive Logic} (or {\em Discursive Logic}) {\bf D2} presented by Stanis\l aw Ja\'{s}kowski in~\cite{jas:48,jas:49}, based on a proposal of his PhD advisor, Jan \L ukasiewicz.} More than this, it introduces the innovative idea of considering, within each calculus $C_n$, a unary (definable) connective $\circ_n$ asserting the {\em well-behavior} (or classical behavior) of a proposition in terms of the explosion law. Namely, in any $C_n$ a contradiction $\{\alpha, \neg\alpha\}$ is not deductively trivial in general (that is, the negation is not {\em explosive}), but any theory containing $\{\alpha, \neg\alpha,\cons_n\alpha\}$ is always deductively trivial (which means that the explosion law is only guaranteed by the conjunction $\alpha \land \neg\alpha \land\cons_n\alpha$). This approach to paraconsistency was afterwards generalized by W. Carnielli and J. Marcos in~\cite{CM} through the notion of {\em logics of formal inconsistency} (\lfis), in which the connective $\cons$ (called {\em consistency operator}) can be a primitive one.

The logics $C_n$ were proved to be non-characterizable by a single finite logical matrix. Moreover, in~\cite{avr:07} it was shown that they are not even characterizable by a  single finite non-deterministic matrix, see Section~\ref{defRNmat}. Despite these results, some decision procedures for the calculi $C_n$ were proposed in the literature, for instance valuations (or bivaluations)~\cite{lop:alv:80} and Fidel structures~\cite{fid:77}. The aim of this paper is presenting a decision procedure for these logics based on the concept of {\em restricted non-deterministic matrices} (RNmatrices). These structures are nothing else than non-deterministic matrices in which the set of valuations is limited to a  suitable subset satisfying certain restrictions.  Moreover, the RNmatrix semantics proposed here induces naturally a labelled tableau system for each $C_n$, which constitutes another decision procedure for these logics.

The paper is organized as follows: in Section~\ref{defRNmat} we introduce the notion of RNmatrices, prove some of their basic properties and give examples of their use. Section~\ref{Cn} recalls the $C$-systems and the notion of \lfis. In Section~\ref{axCL} we show that the limitative result obtained in~\cite{avr:07} concerning Nmatrices can be overcome by means of RNmatrices. In Section~\ref{SectCn}, a $(n+2)$-valued RNmatrix which characterizes the calculus $C_n$ for $n \geq 2$ is defined, constituting a new decision procedure  for these logics. In Section ~\ref{TableauxforCn} we show how the row-branching truth-tables presented as decision methods for $C_n$ for $n \geq 1$ in previous sections naturally induce a tableau system for these logics, which constitutes another decision procedure for them. Finally, in Section~\ref{FinRem} we discuss some possible lines of future research.

\section{From matrices to restricted non-deterministic matrices} \label{defRNmat}

In this section, the notion of {\em restricted non-deterministic matrices}
will be introduced. Previous to this, some basic concepts will be recalled.

A (propositional) signature is a denumerable family $\Theta=(\Theta_n)_{n \geq 0}$ of pairwise disjoint sets; elements of $\Theta_n$ are called {\em $n$-ary connectives}.
The algebra over $\Theta$ freely generated by a denumerable set $\matV=\{p_1,p_2,\ldots\}$ of propositional variables will be denoted by ${\bF}(\Theta,\matV)$. Elements of ${\bF}(\Theta,\matV)$ are called {\em formulas} (over $\Theta$), while elements of the semigroup $Subs(\Theta, \mathcal{V})$ of endomorphisms of ${\bf F}(\Theta, \mathcal{V})$ are called {\em substitutions} (over $\Theta$).

A {\em Tarskian logic} is a pair $\mathcal{L}=(F,\vdash)$ such that  $F$ is a nonempty set and ${\vdash}\subseteq \wp(F)\times F$ (where $\wp(F)$ denotes the powerset of $F$) is a {\em consequence relation} such that:~(i) $\Gamma \vdash \varphi$ whenever $\varphi \in \Gamma$; (ii)~$\Gamma \vdash \varphi$ and $\Gamma \subseteq \Delta$ imply $\Delta \vdash \varphi$; and~(iii)~$\Gamma \vdash \varphi$, for every $\varphi \in \Delta$, and $\Delta \vdash \psi$ imply $\Gamma \vdash \psi$. A logic $\mathcal{L}$ is {\em finitary} if $\Gamma \vdash \varphi$ implies that  $\Gamma_0 \vdash \varphi$ for some finite $\Gamma_0 \subseteq \Gamma$. $\mathcal{L}=({\bF}(\Theta,\matV),\vdash)$ is {\em structural} if $\Gamma \vdash \varphi$ implies $\rho[\Gamma] \vdash \rho(\varphi)$ for every substitution $\rho$ over $\Theta$.\footnote{Along this paper the following notation will  be adopted: if $f:X \to Y$ is a function and $Z \subseteq X$ then $f[Z]$ denotes the set $\{f(x) \ : \ x \in Z\}$.} $\mathcal{L}$ is {\em standard} if it is Tarskian, finitary and structural.

A {\em  logical matrix} over the signature $\Theta$ is a pair $\mathcal{M}=(\mathcal{A}, D)$ such that $\mathcal{A}$ is a $\Theta$-algebra with universe $A$ and $\emptyset \neq D \subseteq A$. A valuation over $\mathcal{M}$ is a homomorphism of $\Theta$-algebras $\nu:{\bf F}(\Theta, \mathcal{V}) \to \mathcal{A}$. The logic associated to $\mathcal{M}$ is defined as follows: $\Gamma \vDash_{\mathcal{M}}\varphi$ iff, for every valuation $\nu$, $\nu(\varphi)\in D$
whenever  $\nu[\Gamma]\subseteq D$. Given a class $\mathbb{M}$ of matrices, the logic associated to $\mathbb{M}$ is given by $\Gamma \vDash_{\mathbb{M}}\varphi$ iff $\Gamma \vDash_{\mathcal{M}}\varphi$ for every $\mathcal{M} \in \mathbb{M}$. Clearly $\vDash_{\mathbb{M}}$ is Tarskian and structural. Moreover, R. W\'ojcicki  has shown in~\cite{woj:70} that every Tarskian and  structural logic is characterized  by a class of logical matrices over its signature.

Matrix logics were generalized by B. Piochi in~ \cite{Piochi,Piochi2} through the notion of  restricted matrices.\footnote{In~ \cite{Piochi2} Piochi uses the name $\mathcal{E}$-matrix instead of restricted matrix.} A {\em restricted logical matrix} (or {\em Rmatrix}) over a signature $\Theta$ is a triple $\mathcal{M}=(\mathcal{A}, D, \mathcal{F})$ such that $(\mathcal{A}, D)$ is a logical matrix over $\Theta$ and $\mathcal{F}$ is a set of valuations over $\mathcal{M}$. If $\nu \circ \rho \in \mathcal{F}$, for every $\nu \in \mathcal{F}$ and any substitution $\rho$, then the Rmatrix is said to be {\em structural}.  We say that $\Gamma$ proves $\varphi$ according to an Rmatrix $\mathcal{M}$,  written as $\Gamma\vDash_{\mathcal{M}}^\mathsf{R}\varphi$, if $\nu(\varphi)\in D$ whenever  $\nu[\Gamma]\subseteq D$, for every valuation $\nu\in\mathcal{F}$. The consequence relation generated by a class of restricted logical matrices is defined as in the case of logical matrices.  

The class of logics generated by structural Rmatrix semantics coincides with the class of Tarskian and structural logics (\cite{Piochi,Piochi2}). However, Rmatrices are more powerful than ordinary logical matrices in the following sense: any Tarskian and structural logic is characterized by a single (but possibly infinite) structural Rmatrix.  

Finite logical matrices result in straightforward decision methods for their respective logics through truth-tables. It becomes then natural to ponder whether all Tarskian and structural logics may be indeed characterized by finite logical matrices. There are well-known negative results to this problem for  several non-classical logics: in 1932 K. G\"odel proved that propositional intuitionistic logic cannot be characterized by a single finite logical matrix (see~\cite{god:32}). His proof was adapted by J. Dugundji to prove the same kind of result for any modal logic between {\bf S1} and {\bf S5} (see~\cite{dug:40}). Uncharacterizability results by a single logical matrix were also obtained  for several \lfis\  (see for instance~\cite{avr:05b,avr:07,CCM,CC16, Avron:19}). To overcome this difficulty in the specific case of \lfis, and in order to obtain a useful decision procedure for these logics, A. Avron and I. Lev introduced in~\cite{avr:lev:01} (see also~\cite{avr:lev:05})  the notion of non-deterministic matrices (or Nmatrices). These structures generalize logical matrices by taking multialgebras (a.k.a. hyperalgebras) instead of algebras.\footnote{Previous to Avron and Lev's work, the use of non-deterministic matrices in logic was already proposed in the literature: N. Rescher's~\cite{res:62} non-deterministic implication and J. Ivlev's non-normal modal systems~\cite{ivl:73,ivl:88} constitute explicit antecedents of this notion.}

\begin{defi} \label{defNmatrix}
(1) Given a signature $\Theta$, a pair $\mathcal{A}=(A,\{\sigma_{\mathcal{A}}\}_{\sigma\in\Theta})$ is said to be a {\em $\Theta$-multialgebra with universe the set $A$} if, given $\sigma \in \Theta_n$, $\sigma_{\mathcal{A}}$ is a function $\sigma_{\mathcal{A}}:A^{n}\rightarrow\wp(A)\setminus\{\emptyset\}$; in particular, $\emptyset\neq\sigma_{\mathcal{A}}\subseteq A$ if $\sigma \in \Theta_0$.\\[1mm]
(2) An {\em homomorphism} between $\Theta$-multialgebras $\mathcal{A}=(A, \{\sigma_{\mathcal{A}}\}_{\sigma\in\Theta})$ and $\mathcal{B}=(B, \{\sigma_{\mathcal{B}}\}_{\sigma\in\Theta})$ is a function $h:A\rightarrow B$ such that, for any $\sigma$ in $\Theta$ of arity $n$ and any $a_{1}, \ldots , a_{n}\in A$,
$h[\sigma_{\mathcal{A}}(a_{1}, \ldots , a_{n})]\subseteq \sigma_{\mathcal{B}}(h(a_{1}), \ldots , h(a_{n}))$.\\
(3) A {\em non-deterministic matrix} (or a {\em Nmatrix}) over $\Theta$ is a pair $\mathcal{M}=(\mathcal{A}, D)$ such that $\mathcal{A}$ is a $\Theta$-multialgebra with universe $A$ and $\emptyset \neq D \subseteq A$. A valuation over $\mathcal{M}$ is a homomorphism of $\Theta$-multialgebras $\nu:{\bf F}(\Theta, \mathcal{V}) \to \mathcal{A}$ (where the $\Theta$-algebra ${\bf F}(\Theta, \mathcal{V})$ is considered as a $\Theta$-multialgebra). The consequence relation associated to a Nmatrix (as well as to a class of Nmatrices) is defined as in the case of logical matrices, but now by using valuations over Nmatrices.
\end{defi}

\begin{remark}
(1)  Observe that,  for every valuation $\nu$, $n$-ary connective $\sigma$ and formulas $\varphi_1,\ldots,\varphi_n$,  $\nu(\sigma(\varphi_1,\ldots,\varphi_n)) \in \sigma_\mathcal{A}(\nu(\varphi_1),\ldots,\nu(\varphi_n))$. Valuations for Nmatrices, in this presentation, have also been known as {\em legal valuations}; however, by presenting them as homomorphisms, it becomes more clear that the notion of Nmatrix semantics corresponds to an exact generalization of the notion of matrix semantics, moving from algebras to multialgebras.\\[1mm]
(2)  The category of multialgebras allows to consider Nmatrices as  simply being multialgebras, that is,  objects of the category: indeed, an Nmatrix $\mathcal{M}=(\mathcal{A}, D)$ over $\Theta$ is nothing more that a multialgebra   $\mathcal{A}^\top=(A,\{\sigma_{\mathcal{A}^\top}\}_{\sigma\in\Theta^\top})$ over the signature $\Theta^\top$ obtained from $\Theta$ by addition of a new constant $\top \in \Theta_0^\top$, such that $\sigma_{\mathcal{A}^\top}=\sigma_{\mathcal{A}}$ for every $\sigma \in \Theta$ and $\top_{\mathcal{A}^\top}=D$. This contrasts with the case of logical matrices, which in general cannot be considered as being algebras (unless the set of designated values is a singleton).
\end{remark}

\noindent
In~\cite[Theorem~11]{avr:07} Avron has shown that some \lfis\ (including da Costa's system $C_1$) cannot be characterized by a single finite Nmatrix, so establishing a Dugundji-like theorem with respect to Nmatrices. He also defines an infinite characteristic Nmatrix for each of these logics. Such infinite Nmatrices are effective, thus inducing a decision procedure for these logics (see~\cite{avr:07,Avron:19,Avron:Arieli:Zamansky:18}). However, these procedures can require the use of too many truth-values (see Example~\ref{ex3valC1}). As an alternative solution to the decidability problem of such logics,  we propose an additional generalization of Nmatrices, the {\em restricted non-deterministic matrices} (or {\em RNmatrices}), which combines Nmatrices and Rmatrices paradigms. Thus, in Section~\ref{axCL} it will be obtained a three-valued RNmatrices for $C_1$ and for a subsystem of it called \mbccl, which also lies in the scope of Avron's uncharacterizability result mentioned above. Moreover, in Section~\ref{SectCn} the three-valued characteristic RNmatrix for $C_1$ will be generalized to a $(n+2)$-valued characteristic RNmatrix for $C_n$ for $n \geq 2$. This constitutes a relatively simple decision procedure for da Costa's hierarchy through row-branching truth-tables or, alternatively, by tableaux semantics, as shown in Section~\ref{TableauxforCn}. 

\begin{defi}
A {\em restricted non-deterministic matrix}, or {\em restricted Nmatrix} or  simply an {\em RNmatrix}, over a signature $\Theta$ is a triple $\mathcal{M}=(\mathcal{A}, D, \mathcal{F})$ such that:
\begin{enumerate}
\item $(\mathcal{A}, D)$ is a non-deterministic matrix over $\Theta$;
\item $\mathcal{F}$ is a subset of the set of valuations over $(\mathcal{A}, D)$.
\end{enumerate}
An RNmatrix $\mathcal{M}=(\mathcal{A}, D, \mathcal{F})$ is  {\em structural} if  $\nu\circ\rho\in\mathcal{F}$, for every $\nu\in\mathcal{F}$ and any substitution $\rho$ over $\Theta$. The consequence relation with respect to an RNmatrix $\mathcal{M}=(\mathcal{A}, D, \mathcal{F})$, denoted by $\vDash_{\mathcal{M}}^\mathsf{RN}$, is defined as in the case of Rmatrices: $\Gamma\vDash_{\mathcal{M}}^\mathsf{RN}\varphi$ if $\nu(\varphi)\in D$ whenever  $\nu[\Gamma]\subseteq D$, for every valuation $\nu\in\mathcal{F}$. As in the case of logical matrices and Rmatrices, ${\vDash_{\mathbb{M}}^\mathsf{RN}} = \bigcap_{\mathcal{M} \in \mathbb{M}} {\vDash_{\mathcal{M}}^\mathsf{RN}}$ for any nonempty class $\mathbb{M}$ of RNmatrices.
\end{defi}

\begin{theorem}
	Given a nonempty class $\mathbb{M}$ of structural RNmatrices, $\vDash_{\mathbb{M}}^\mathsf{RN}$ is Tarskian and structural. Moreover, any Tarskian and structural logic is characterized by a single structural RNmatrix.
\end{theorem}
\begin{proof}
	The first part is obvious. The second part follows from the fact that any Rmatrix is an RNmatrix.
\end{proof}

\noindent
However, RNmatrices are stronger than Rmatrices, in the following sense:

\begin{theorem}
Every Tarskian logic of the form $\mathcal{L}=({\bf F}(\Theta, \mathcal{V}), \vdash)$ is characterizable by a two-valued RNmatrix. If $\mathcal{L}$  is structural, so is the RNmatrix.  
\end{theorem}
\begin{proof}
Consider the $\Theta$-multialgebra $\textbf{2}(\Theta)$ with universe $\{0,1\}$ and all the multioperations returning the whole universe. We then define $\mathcal{F}_{\mathcal{L}}$ as the set of valuations $\nu:{\bf F}(\Theta, \mathcal{V})\rightarrow\textbf{2}(\Theta)$ such that there exists a $\mathcal{L}$-closed set of formulas $\Gamma$ over $\Theta$ for which $\nu(\gamma)=1$ iff $\gamma\in\Gamma$ (notice that every function from ${\bf F}(\Theta, \mathcal{V})$ to $\textbf{2}(\Theta)$ is a homomorphism). Consider the RNmatrix $\textbf{2}(\mathcal{L})=(\textbf{2}(\Theta), \{1\}, \mathcal{F}_{\mathcal{L}})$. It is straightforward to prove that $\Gamma\vdash\varphi$ iff $\Gamma\vDash_{\textbf{2}(\mathcal{L})}\varphi$. Clearly, if $\mathcal{L}$ is structural, so is $\textbf{2}(\mathcal{L})$. 
\end{proof}

\begin{remark} It is worth noting that the notion of RNmatrices is not new, and it was already considered in the literature. For instance, J. Kearns defines in~\cite{kear:81} four-valued  RNmatrices which characterize  some normal modal logics (see Example~\ref{Kearns} below). Avron and Konikowska have also considered the restricted version of any Nmatrix trough their {\em static semantics} (see Example~\ref{static} below).  More recently, RNmatrices were  also considered by P. Pawlowski and R. Urbaniak in the context of logics of informal provability  (see~\cite{paw:urb:18,Pawlowski}), and by H. Omori and D. Skurt, in the context of modal logics (see~\cite{OS:20}). 	
\end{remark}

\begin{example}[Kearns's semantics for modal logics] \label{Kearns}
As a way to overcome the limitations imposed by Dugundji's theorem, J. Kearns proposed in~\cite{kear:81} a four-valued Nmatrix semantics for the modal logics {\bf T},  {\bf S4}  and {\bf S5}.\footnote{Kearns and Ivlev's proposals were extended independently in~\cite{con:far:per:15,con:far:per:16} and~\cite{omo:sku:16}.} However, not all the valuations over such Nmatrices must be considered, but a special subset of them.  This set of valuations can be recast as follows:\footnote{Here, we are following the presentation of Kearns's approach given in~\cite{con:far:per:15}.} let 
$\{T,t,f,F\}$ be the domain of the Nmatrix $\mathcal{M}_{\bf L}$ proposed by Kearns for the modal system ${\bf L} \in \{{\bf T},  {\bf S4}, {\bf S5}\}$,  where $D = \{T,t\}$ is the set of designated values, and let $\val^{\bf L}$ be the set of valuations over $\mathcal{M}_{\bf L}$ (recall Definition~\ref{defNmatrix}(3)). Let $\val_k^{\bf L} \subseteq \val^{\bf L}$ be defined as follows: $\val_0^{\bf L}=\val^{\bf L}$ and, for every $k \geq 0$, 
$$\val_{k+1}^{\bf L}= \{\nu \in  \val_k^{\bf L} \ :  \ \mbox{for every
	formula $\alpha$, $\val_k^{\bf L}(\alpha) \subseteq D$ implies $\nu(\alpha)=
	T$} \}$$ 
where $\val_k^{\bf L}(\alpha)= \{ \nu(\alpha) \ : \ \nu \in
\val_k^{\bf L}\}$, for every $k$ and $\alpha$. Finally, the set of {\bf L}-valuations is given by $\mathcal{F}_{\bf L} =
\bigcap_{k \geq 0} \val_k^{\bf L}$.  Then, $\vdash_{\bf L} \alpha$ iff $\nu(\alpha)=T$ for every $\nu \in \mathcal{F}_{\bf L}$.  Clearly, Kearns's semantics for {\bf L} corresponds to the semantics given by the RNmatrix $\mathcal{K}_{\bf L}=(\mathcal{A}_{\bf L},\{T\},\mathcal{F}_{\bf L})$, where $\mathcal{A}_{\bf L}$ is the multialgebra underlying the Nmatrix $\mathcal{M}_{\bf L}$.
Moreover, it is easy to see that $\mathcal{K}_{\bf L}$ is structural. Indeed, by induction on $k$ it can be proved that, for every valuation $\nu$ and every substitution $\rho$, $\nu \circ \rho \in \val_k^{\bf L}$ whenever $\nu \in \val_k^{\bf L}$. The case $k=0$ is clearly true. Suppose that, for every valuation $\nu$ and every substitution $\rho$, $\nu \circ \rho \in \val_{k}^{\bf L}$ whenever $\nu \in \val_{k}^{\bf L}$ (IH). Let $\nu \in \val_{k+1}^{\bf L}$, and let $\rho$ be a substitution. By definition, $\nu \in \val_{k}^{\bf L}$ and so $\nu \circ \rho \in \val_{k}^{\bf L}$, by (IH). Let $\alpha$ be a formula such that $\val_k^{\bf L}(\alpha) \subseteq D$. If $\nu' \in \val_{k}^{\bf L}$ then $\nu' \circ \rho \in \val_{k}^{\bf L}$, by (IH), hence $\nu' (\rho(\alpha))=\nu' \circ \rho (\alpha) \in D$; that is, $\val_k^{\bf L}(\rho(\alpha)) \subseteq D$. This implies that $\nu \circ \rho (\alpha)=\nu(\rho(\alpha)) =T$, since $\nu \in \val_{k+1}^{\bf L}$. From this,  $\nu \circ \rho \in \val_{k+1}^{\bf L}$. This proves that  $\nu \circ \rho \in \val_k^{\bf L}$ whenever $\nu \in \val_k^{\bf L}$, for every $k\geq 0$. Finally, let $\nu \in \mathcal{F}_{\bf L}$ and let $\rho$ be a substitution. Given $k \geq 0$, $\nu \in \val_k^{\bf L}$ and so $\nu \circ \rho \in \val_k^{\bf L}$. This means that $\nu \circ \rho \in \mathcal{F}_{\bf L}$, that is, $\mathcal{K}_{\bf L}$ is structural.
\end{example}

\begin{example} [Static semantics for Nmatrices] \label{static}
The  valuations over Nmatrices considered by Avron and Lev produce what is called {\em dynamic semantics} over Nmatrices. Avron and Konikowska (\cite{AK:05}) have also considered a restriction of the usual valuations, the {\em static semantics}. Given an Nmatrix $\mathcal{M}$, its static semantics is given by the set $\mathcal{F}^s_\mathcal{M}$ of valuations $\nu$ over $\mathcal{M}$ such that, for all formulas $\sigma(\varphi_1,\ldots,\varphi_n)$ and $\sigma(\psi_1,\ldots,\psi_n)$ (where $\sigma$ is an $n$-ary connective), $\nu(\sigma(\varphi_1,\ldots,\varphi_n))=\nu(\sigma(\psi_1,\ldots,\psi_n))$ provided that $\nu(\varphi_i)=\nu(\psi_i)$ for every $1 \leq i \leq n$. Clearly, $(\mathcal{A},D,\mathcal{F}^s_\mathcal{M})$ is a structural RNmatrix for every Nmatrix $\mathcal{M}=(\mathcal{A},D)$.
\end{example}

\begin{example}[PNmatrices]\label{PNmatrices}
PNmatrices, first defined in \cite{Baaz:13}, also generalize Nmatrices by allowing for partial multialgebras instead of simply multialgebras. We will use in this example, however, the definition to be found in \cite{CM:19}: given a signature $\Theta$, a $\Theta$-partial multialgebra is a pair $\mathcal{A}=(A, \{\sigma_{\mathcal{A}}\}_{\sigma\in\Theta})$ such that, if $\sigma\in\Theta_{n}$, $\sigma_{\mathcal{A}}$ is a function from $A^{n}$ to $\wp(A)$; a valuation for $\mathcal{A}$ is then a map $\nu: {\bf F}(\Theta, \mathcal{V})\rightarrow A$ such that, for $\sigma\in\Theta$ of arity $n$ and $\alpha_{1}, \ldots  , \alpha_{n}\in {\bf F}(\Theta, \mathcal{V})$, $\nu(\sigma(\alpha_{1}, \ldots  , \alpha_{n}))\in \sigma_{\mathcal{A}}(\nu(\alpha_{1}), \ldots  , \nu(\alpha_{n}))$;\footnote{Observe that this presupposes that the set $\sigma_{\mathcal{A}}(\nu(\alpha_{1}), \ldots  , \nu(\alpha_{n}))$ is nonempty; otherwise, there is no such valuation, given that they are assumed to be total functions.} finally, for a pair $\mathcal{M}=(\mathcal{A}, D)$, with $\mathcal{A}$ a $\Theta$-partial multialgebra and $D$ a subset of the universe of $\mathcal{A}$, $\Gamma\vDash_{\mathcal{M}}\varphi$ iff, for every valuation $\nu$ for $\mathcal{A}$, $\nu[\Gamma]\subseteq D$ implies $\nu(\varphi)\in D$. Now, consider the $\Theta$-multialgebra $\mathcal{A}^{\emptyset}=(A\cup\{o\}, \{\sigma_{\mathcal{A}^{\emptyset}}\}_{\sigma\in\Theta})$, where we assume $o\notin A$, such that $\sigma_{\mathcal{A}^{\emptyset}}(a_{1}, \ldots  , a_{n})$ equals $\sigma_{\mathcal{A}}(a_{1}, \ldots  , a_{n})$, if the last set is not empty and $a_{1}, \ldots  , a_{n}\in A$, and $\{o\}$ otherwise. Then, the set of valuations for $\mathcal{A}$ is the set of homomorphisms from ${\bf F}(\Theta, \mathcal{V})$ to $\mathcal{A}^{\emptyset}$ which do not have $o$ in their range. Defining $\mathcal{M}^{\emptyset}=(\mathcal{A}^{\emptyset}, D, \mathcal{F})$, for 
\[\mathcal{F}=\{\nu:{\bf F}(\Theta, \mathcal{V})\rightarrow\mathcal{A}^{\emptyset}\ : \ \nexists\alpha\in {\bf F}(\Theta, \mathcal{V})\text{ such that }\nu(\alpha)=o\},\]
we obtain that $\mathcal{M}^{\emptyset}$ induces precisely the same deductive operator as $\mathcal{M}$.
\end{example}

\noindent
In Subsection~\ref{FCn}  another historical antecedent of RNmatrices  will be discussed: the Fidel structures semantics.

\section{da Costa's Calculi $C_n$, and other \lfis} \label{Cn}
 
 In this section it will be recalled the da Costa's hierarchy of paraconsistent calculi $C_n$, as well as some \lfis\ which are relevant to our discussion. 

Let $\Sigma$ be the propositional signature for the calculi $C_n$ such that $\Sigma_1=\{\neg\}$, $\Sigma_2=\{\land,\lor,\to\}$, and $\Sigma_k=\emptyset$ otherwise.

Consider the following abbreviations in ${\bF}(\Sigma,\matV)$: $\alpha^0 := \alpha$ and $\alpha^{n+1}:=\neg(\alpha^n \land \neg(\alpha^n))$ for every $0 \leq n < \omega$. On the other hand, $\alpha^{(0)} := \alpha$, $\alpha^{(1)}:=\alpha^1$ and $\alpha^{(n+1)}:=\alpha^{(n)} \land \alpha^{n+1}$ for every $1 \leq n < \omega$. The formula $\alpha^1=\neg(\alpha \land \neg\alpha)$ will be also denoted by $\alpha^\circ$. Accordingly, $\alpha^k$ can be alternatively denoted by $\alpha^{\circ \cdots \circ}$, where $\circ \cdots \circ$ denotes a sequence of $k$ iterations of $\circ$, for $k \geq 2$.

\begin{defi} [The calculi $C_n$, for $n \geq 1$] \label{hilCn} Let $n \geq 1$. The logic $C_n$ is defined over the signature $\Sigma$ by the following Hilbert  calculus:\\[2mm]
	{\bf Axiom schemata:}\vspace*{-5mm}
	\begin{gather}
	\alpha \imp \big(\beta \imp \alpha\big)             \tag{\kax} \\
	\Big(\alpha\imp\big(\beta\imp\gamma\big)\Big) \imp
	\Big(\big(\alpha\imp\beta\big)\imp\big(\alpha\imp\gamma\big)\Big)
	\tag{\axTrans}\\
	\alpha \imp \Big(\beta \imp \big(\alpha \land \beta\big)
	\Big)  \tag{\axed}\\
	\big(\alpha \land \beta\big) \imp \alpha         \tag{\axeea}\\	
	\big(\alpha \land \beta\big) \imp \beta          \tag{\axeeb}\\
	\alpha \imp \big(\alpha \lor \beta\big)          \tag{\axouda}\\
	\beta \imp \big(\alpha \lor \beta\big)           \tag{\axoudb}\\
	\Big(\alpha \imp \gamma\Big) \imp \Big(
	(\beta \imp \gamma) \imp
	\big(
	(\alpha \lor \beta) \imp \gamma
	\big)\Big)                               \tag{\axoue}\\
	\alpha \lor \lnot \alpha                        \tag{\axtnd}\\
	\neg\neg \alpha \imp \alpha
	\tag{\axcf}\\
	\alpha^{(n)} \imp \Big(\alpha \imp \big(\lnot \alpha \imp \beta\big)\Big)
	\tag{\axexp$_n$}\\
	(\alpha^{(n)} \land \beta^{(n)}) \imp \big((\alpha \land \beta)^{(n)} \land (\alpha \lor \beta)^{(n)} \land (\alpha \to \beta)^{(n)}\big)
	\tag{\axp$_n$}
	\end{gather}	
	{\bf Inference rule:}
	\[\frac{\alpha \ \ \ \ \alpha\imp
		\beta}{\beta}  \tag{\MP}\]
\end{defi}

\begin{remark} \label{remCn}
The original presentation of da Costa (see~\cite{dC63}) considers, instead of axiom (\axexp$_n$), the following one:
\begin{gather}
\alpha^{(n)} \imp \big((\beta \to \alpha) \to ((\beta \to \neg\alpha) \to \neg\beta)\big)
\tag{\axdc$_n$}
\end{gather}
It is easy to show the equivalence of both presentations of $C_n$. In adition, it is well-known that the Dummett law $\alpha \vee (\alpha \imp \beta)$ is derivable in every $C_n$.
\end{remark}

\begin{defi} \label{bival-def}
A {\em bivaluation} for $C_n$ (or a {\em $C_n$-bivaluation}) is a function $\mathsf{b}:{\bF}(\Sigma,\matV) \to {\bf 2}$  (where ${\bf 2}:=\{0,1\}$) satisfying the following clauses:\\[2mm]
$\begin{array}{ll}
(B1) & \mathsf{b}(\alpha \land \beta)=1  \ \ \mbox{ iff } \ \  \mathsf{b}(\alpha)=1 \ \mbox{ and } \ \mathsf{b}(\beta)=1;\\[1mm]
(B2) & \mathsf{b}(\alpha \lor \beta)=1  \ \ \mbox{ iff } \ \  \mathsf{b}(\alpha)=1 \ \mbox{ or } \ \mathsf{b}(\beta)=1;\\[1mm]
(B3) & \mathsf{b}(\alpha \to \beta)=1  \ \ \mbox{ iff } \ \  \mathsf{b}(\alpha)=0 \ \mbox{ or } \ \mathsf{b}(\beta)=1;\\[1mm]
(B4) & \mathsf{b}(\alpha)=0  \ \ \mbox{ implies that } \ \  \mathsf{b}(\neg\alpha)=1;\\[1mm]
(B5) & \mathsf{b}(\neg\neg\alpha)=1  \ \ \mbox{ implies that } \ \  \mathsf{b}(\alpha)=1;\\[1mm]
(B6)_n & \mathsf{b}(\alpha^{n-1})=\mathsf{b}(\neg (\alpha^{n-1}))  \ \ \mbox{ iff } \ \  \mathsf{b}(\alpha^n)=0;\\[1mm]
(B7) & \mathsf{b}(\alpha)=\mathsf{b}(\neg \alpha)  \ \ \mbox{ iff } \ \  \mathsf{b}(\neg(\alpha^\circ))=1;\\[1mm]
(B8) & \mathsf{b}(\alpha)\neq \mathsf{b}(\neg \alpha) \ \mbox{ and } \  \mathsf{b}(\beta)\neq \mathsf{b}(\neg \beta) \ \ \mbox{ implies that }\ \ \\
& \mathsf{b}(\alpha\#\beta)\neq \mathsf{b}(\neg (\alpha\#\beta)), \ \ \mbox{for } \# \in \{\land,\lor,\to\}.\\[1mm]
\end{array}$
\end{defi}
 
\noindent The semantical consequence relation  w.r.t. bivaluations for $C_n$ will be denoted by $\vDash_{n}^{\bf 2}$. Namely, $\Gamma \vDash_{n}^{\bf 2} \varphi$ iff $\mathsf{b}(\varphi)=1$, for any $C_n$-bivaluation $\mathsf{b}$ such that $\mathsf{b}[\Gamma]\subseteq\{1\}$. Although bivaluations for da Costa's hierarchy were first proposed in~\cite{dCA:77}, the first correct proof that the method indeed worked appeared in 1980, presented by Lopari\'c and Alves~\cite{lop:alv:80}.   
 
\begin{theorem} [Soundness and completeness of $C_n$ w.r.t. bivaluations, \cite{lop:alv:80}] \label{comple-Cn-biv} 
Fix $n \geq 1$, and let $\Gamma \cup \{\varphi\} \subseteq {\bF}(\Sigma,\matV)$. Then: $\Gamma \vdash_{C_n} \varphi$ \ iff \ $\Gamma \vDash_{n}^{\bf 2} \varphi$. 
\end{theorem}

\begin{remark} \label{bival-RN}  It should be noticed that bivaluation  semantics corrresponds to a  structural RNmatrix. Indeed, consider the Nmatrix over $\Sigma$ with domain $\{0,1\}$, deterministic operations for $\vee$, $\wedge$ and $\to$ (the classical truth-tables), the multioperation $\tilde{\neg}(0)=\{1\}$ and $\tilde{\neg}(1)=\{0,1\}$ and $1$ as the only designated value.Then, the set $\mathcal{F}$ of bivaluations is a set of valuations over such Nmatrix such that $\nu\circ\rho\in\mathcal{F}$, for every $\nu\in\mathcal{F}$ and any substitution $\rho$ over $\Sigma$. Thus, bivaluations introduced in~\cite{dCA:77} and~\cite{lop:alv:80}, together with the induced {\em quasi-matrices}, constitute one of the early examples of RNmatrix semantics.
\end{remark}

As discussed in the Introduction,  da Costa's approach to paraconsistency was generalized through the notion of \lfis. An interesting class of them is defined over the signature $\Sigma^{\circ}$ obtained from $\Sigma$ by adding an additional unary connective $\circ$ to express `consistency' in the sense of respecting the explosion law. That is, every formula follows from $\{\alpha, \neg\alpha, \circ \alpha\}$, despite $\{\alpha, \neg\alpha\}$ being not necessarily trivial. The basic \lfi\ studied in~\cite{CCM} is \mbC, obtained from the axiom schemata $\textbf{Ax1}$ trough $\textbf{Ax9}$ for $C_n$ by addition of the Dummett law $\alpha \vee (\alpha \imp \beta)$ (recall Remark~\ref{remCn}) and 
\[\tag{\textbf{bc1}}\circ\alpha\rightarrow(\alpha\rightarrow(\neg \alpha\rightarrow \beta)),\]
together with  Modus Ponens as the only inference rule. The logics $\mbcci$ and $\mbccl$ were also considered in~\cite{CC16}, being obtained from \mbc\ by addition, respectively, of the axiom schema $\textbf{ci}$: $\neg\cons\alpha\rightarrow(\alpha\wedge\neg \alpha)$ and $\textbf{cl}$: $\neg(\alpha\wedge\neg \alpha)\rightarrow\circ\alpha$.\footnote{These logics were already considered in~\cite{avr:07} under the names of  {\bf Bi} and  {\bf Bl}, respectively. The former was originally presented in~\cite{avr:05} by means of a sequent calculus called {\bf B}[\{{\bf i1},{\bf i2}\}].} The logic $\CILA$, proposed in~\cite{CM}, is obtained from $\mbccl$ by addition of the axiom schemata $\textbf{ci}$, $\textbf{cf}$: $\neg\neg\alpha\rightarrow\alpha$ and 
\[\tag{$\textbf{ca}_{\#}$}(\circ\alpha\wedge\circ\beta)\rightarrow\circ(\alpha\#\beta),\quad\text{for}\quad\#\in\{\vee, \wedge, \rightarrow\}.\]
It is a well-known result (\cite{CM}) that $\CILA$ and $C_1$ are equivalent systems.

\begin{defi}
A {\em bivaluation} for $\mbC$ is a function $\mathsf{b}:{\bF}(\Sigma^{\circ},\matV) \to {\bf 2}$ satisfying clauses $(B1)$ trough $(B4)$ from Definition~\ref{bival-def}, plus:\\[1mm]
$\begin{array}{ll}
(B^\prime1) & \mathsf{b}(\circ\alpha)=1  \ \ \mbox{ implies that } \ \  \mathsf{b}(\alpha)=0 \ \mbox{ or } \ \mathsf{b}(\neg\alpha)=0.\\[1mm]
\end{array}$\\
A bivaluation for $\mbcci$ or $\mbccl$ is a bivaluation for $\mbC$ satisfying additionally that, respectively,\\[1mm]
$\begin{array}{ll}
(B^\prime2) & \mathsf{b}(\neg\cons\alpha)=1  \ \ \mbox{ implies that } \ \  \mathsf{b}(\alpha)=1 \ \mbox{ and } \ \mathsf{b}(\neg\alpha)=1;\\[1mm]
(B^\prime3) & \mathsf{b}(\neg(\alpha\wedge\neg\alpha))=1  \ \ \mbox{ implies that } \ \  \mathsf{b}(\circ\alpha)=1.\\[1mm]
\end{array}$\\
A bivaluation for $\CILA$ is a bivaluation for both $\mbcci$ and $\mbccl$, satisfying additionally conditions $(B5)$ and $(B8)$  from Definition~\ref{bival-def}.
\end{defi}

\noindent
It is a well-known fact, of which one can find a proof in \cite{CC16}, that their respective bivaluations can characterize each of the logics $\mbC$, $\mbcci$, $\mbccl$ and $\CILA$; we will denote the semantical consequence relation, with respect to bivaluations, for any logic $\mathcal{L}$ among these four, by $\vDash^{\textbf{2}}_{\mathcal{L}}$. As observed in Remark~\ref{bival-RN} for $C_1$, they correspond to RNmatrices over $\{0,1\}$.

\section{A solution to a Dugundji-like theorem w.r.t. Nmatrices} \label{axCL}

Recall from Section~\ref{defRNmat} that Avron obtained in~\cite[Theorem~11]{avr:07} a Dugundji-like theorem w.r.t. Nmatrices for some \lfis, including \mbccl\ and \cila\ (and so for da Costa's $C_1$). That is, none of these logics can be characterized by a single finite Nmatrix. The problem arises specifically from axiom $\cl$.
In this section we present three-valued  RNmatrices which characterize \mbccl\ and \cila, showing that RNmatrices improve drastically the expressive power of Nmatrices in this specific sense and allow one to define simple and elegant decision procedures for several logics which cannot be characterized by means of a finite Nmatrix.

\subsection{The case of \mbCcl}

Consider the $\Sigma^{\circ}$-multialgebra $\mathcal{A}_{\mbccl}$ with universe $\{F, t, T\}$ and multioperations given by the tables below, where $D=\{t, T\}$ and $U=\{F\}$.
\begin{figure}[H]
\centering
\begin{minipage}[t]{4.5cm}
\centering
\begin{tabular}{|l|c|c|r|}
\hline
$\tilde{\vee}$ & $F$ & $t$ & $T$ \\ \hline
$F$ & $U$ & $D$ & $D$ \\ \hline
$t$ & $D$ & $D$ & $D$ \\ \hline
$T$ & $D$ & $D$ & $D$ \\ \hline
\end{tabular}
\end{minipage}
\centering
\begin{minipage}[t]{4.5cm}
\centering
\begin{tabular}{|l|c|c|r|}
\hline
$\tilde{\wedge}$ & $F$ & $t$ & $T$ \\ \hline
$F$ & $U$ & $U$ & $U$ \\ \hline
$t$ & $U$ & $D$ & $D$ \\ \hline
$T$ & $U$ & $D$ & $D$ \\ \hline
\end{tabular}
\end{minipage}
\end{figure}
\vspace*{-5mm}
\begin{figure}[H]
\centering
\begin{minipage}[t]{3cm}
\centering
\begin{tabular}{|l|r|}
\hline
 & $\tilde{\neg}$ \\ \hline
$F$ & $D$\\ \hline
$t$ & $D$\\ \hline
$T$ & $U$ \\ \hline
\end{tabular}
\end{minipage}
\begin{minipage}[t]{5cm}
\centering
\begin{tabular}{|l|c|c|r|}
\hline
$\tilde{\rightarrow}$ & $F$ & $t$ & $T$ \\ \hline
$F$ & $D$ & $D$ & $D$ \\ \hline
$t$ & $U$ & $D$ & $D$ \\ \hline
$T$ & $U$ & $D$ & $D$ \\ \hline
\end{tabular}
\end{minipage}
\begin{minipage}[t]{3cm}
\centering
\begin{tabular}{|l|r|}
\hline
 & $\tilde{\circ}$ \\ \hline
$F$ & $D$\\ \hline
$t$ & $U$\\ \hline
$T$ & $D$ \\ \hline
\end{tabular}
\end{minipage}
\end{figure}

\noindent
Now, let  $\mathcal{M}_{\mbCcl}=(\mathcal{A}_{\mbccl}, D, \mathcal{F}_{\mbCcl})$ be the restricted Nmatrix such that $\mathcal{F}_{\mbCcl}$ is the set of homomorphisms $\nu:{\bf F}(\Sigma^{\circ},\mathcal{V})\rightarrow\mathcal{A}_{\mbCcl}$ satisfying that, if $\nu(\alpha)=t$, then $\nu(\alpha\wedge\neg\alpha)=T$.
It is clear that this RNmatrix is structural, and not difficult to prove that $\mathcal{M}_{\mbCcl}$ models the axiom schemata and inference rule of $\mbCcl$. The following theorem is proved by induction on the length of a derivation.

\begin{theorem} [Soundness of \mbccl\ w.r.t. $\mathcal{M}_{\mbCcl}$]
Let $\Gamma\cup\{\varphi\}$ be a set of formulas of $\mbCcl$. If $\Gamma\vdash_{\mbCcl}\varphi$ then $\Gamma\vDash_{\mathcal{M}_{\mbCcl}}^\mathsf{RN}\varphi$.
\end{theorem}

\noindent
Now, we wish to show completeness, that is: $\Gamma\vDash_{\mathcal{M}_{\mbCcl}}^\mathsf{RN}\varphi$ implies that $\Gamma\vdash_{\mbCcl}\varphi$. It will be shown that, given a bivaluation $\mathsf{b}$ for $\mbCcl$, there exists a valuation $\nu$ which lies in $\mathcal{F}_{\mbCcl}$ such that $\mathsf{b}(\alpha)=1$ if and only if $\nu(\alpha)\in D$. From this, completeness is proved as follows: assuming that $\Gamma\vDash_{\mathcal{M}_{\mbCcl}}^\mathsf{RN}\varphi$, let $\mathsf{b}$ be a bivaluation such that $\mathsf{b}[\Gamma]\subseteq\{1\}$. Then, the valuation $\nu$ obtained from $\mathsf{b}$ is such that $\nu[\Gamma]\subseteq D$, and therefore $\nu(\varphi)\in D$. Hence $\mathsf{b}(\varphi)=1$, proving that $\Gamma\vDash^{\textbf{2}}_{\mbCcl}\varphi$; the latter  implies that $\Gamma\vdash_{\mbCcl}\varphi$, by completeness of \mbccl\ w.r.t. bivaluations. 

So, given a bivaluation  $\mathsf{b}$ for $\mbCcl$, consider the map $\nu:{\bf F}(\Sigma^\circ,\mathcal{V})\rightarrow\{F,t,T\}$ such that:\\
1. $\nu(\alpha)=F \ \Leftrightarrow \ \mathsf{b}(\alpha)=0$ (and so $\mathsf{b}(\neg\alpha)=1$);\\
2. $\nu(\alpha)=t \ \Leftrightarrow \ \mathsf{b}(\alpha)=1$ and $\mathsf{b}(\neg\alpha)=1$;\\
3. $\nu(\alpha)=T \ \Leftrightarrow \ \mathsf{b}(\alpha)=1$ and $\mathsf{b}(\neg\alpha)=0$.

Notice that $\nu$ is well defined and, clearly, $\mathsf{b}(\alpha)=1$ if and only if $\nu(\alpha)\in D$.

\begin{prop}
$\nu$ is a $\Sigma^{\circ}$-homomorphism between ${\bf F}(\Sigma^{\circ}, \mathcal{V})$ and $\mathcal{A}_{\mbCcl}$ which also lies in $\mathcal{F}_{\mbccl}$.
\end{prop}

\begin{proof}
It is easy to prove that $\nu$ is a homomorphism, by analyzing all the possible cases. For instance, if $\nu(\alpha\land\beta) \in D$ then $\mathsf{b}(\alpha\land\beta)=1$ and so $\mathsf{b}(\alpha)=\mathsf{b}(\beta)=1$. Hence $\nu(\alpha),\nu(\beta) \in D$ and so $\nu(\alpha\land\beta)\in D=\nu(\alpha)\,\tilde{\land}\,\nu(\beta)$. The other cases are proved analogously.
To see that $\nu$ is in $\mathcal{F}_{\mbccl}$, assume $\nu(\alpha)=t$, hence $\mathsf{b}(\alpha)=\mathsf{b}(\neg\alpha)=1$. From $(B1)$, $\mathsf{b}(\alpha\wedge\neg\alpha)=1$. If $\mathsf{b}(\neg(\alpha\wedge\neg\alpha))=1$ then $\mathsf{b}(\circ\alpha)=1$, by $(B^\prime3)$ and so, by $(B^\prime1)$, $\mathsf{b}(\alpha)=0$ or $\mathsf{b}(\neg\alpha)=0$,  a contradiction. So, $\mathsf{b}(\neg(\alpha\wedge\neg\alpha))=0$ and then $\nu(\alpha\wedge\neg\alpha)=T$.
\end{proof}

\noindent
From the considerations above, this implies the following:

\begin{theorem} [Completeness of \mbccl\ w.r.t. $\mathcal{M}_{\mbCcl}$]
	Given a set of formulas $\Gamma\cup\{\varphi\}$ of $\mbCcl$, if $\Gamma\vDash_{\mathcal{M}_{\mbCcl}}^\mathsf{RN}\varphi$ then $\Gamma\vdash_{\mbCcl}\varphi$.
\end{theorem}

\subsection{The case of \CILA} \label{Cila}

Consider the previously defined $\Sigma^{\circ}$-multialgebra $\mathcal{A}_{\mbccl}$: we define a submultialgebra $\mathcal{A}_{\CILA}$ of $\mathcal{A}_{\mbccl}$ trough the following tables, where  $D=\{t, T\}$.
\begin{figure}[H]
\centering
\begin{minipage}[t]{4.5cm}
\centering
\begin{tabular}{|l|c|c|r|}
\hline
$\tilde{\vee}$ & $F$ & $t$ & $T$ \\ \hline
$F$ & $\{F\}$ & $D$ & $\{T\}$ \\ \hline
$t$ & $D$ & $D$ & $D$ \\ \hline
$T$ & $\{T\}$ & $D$ & $\{T\}$ \\ \hline
\end{tabular}
\end{minipage}
\centering
\begin{minipage}[t]{4.5cm}
\centering
\begin{tabular}{|l|c|c|r|}
\hline
$\tilde{\wedge}$ & $F$ & $t$ & $T$ \\ \hline
$F$ & $\{F\}$ & $\{F\}$ & $\{F\}$ \\ \hline
$t$ & $\{F\}$ & $D$ & $D$ \\ \hline
$T$ & $\{F\}$ & $D$ & $\{T\}$ \\ \hline
\end{tabular}
\end{minipage}
\end{figure}
\vspace*{-5mm}
\begin{figure}[H]
\centering
\begin{minipage}[t]{3cm}
\centering
\begin{tabular}{|l|r|}
\hline
 & $\tilde{\neg}$ \\ \hline
$F$ & $\{T\}$\\ \hline
$t$ & $D$\\ \hline
$T$ & $\{F\}$ \\ \hline
\end{tabular}
\end{minipage}
\begin{minipage}[t]{5cm}
\centering
\begin{tabular}{|l|c|c|r|}
\hline
$\tilde{\rightarrow}$ & $F$ & $t$ & $T$ \\ \hline
$F$ & $\{T\}$ & $D$ & $\{T\}$ \\ \hline
$t$ & $\{F\}$ & $D$ & $D$ \\ \hline
$T$ & $\{F\}$ & $D$ & $\{T\}$ \\ \hline
\end{tabular}
\end{minipage}
\begin{minipage}[t]{3cm}
\centering
\begin{tabular}{|l|r|}
\hline
 & $\tilde{\circ}$ \\ \hline
$F$ & $\{T\}$\\ \hline
$t$ & $\{F\}$\\ \hline
$T$ & $\{T\}$ \\ \hline
\end{tabular}
\end{minipage}
\end{figure}
 Let $\mathcal{M}_{\cila}=(\mathcal{A}_{\cila}, D, \mathcal{F}_{\cila})$ be  the restricted Nmatrix  where $\mathcal{F}_{\cila}$ is the set of homomorphisms $\nu:{\bf F}(\Sigma^{\circ}, \mathcal{V})\rightarrow\mathcal{A}_{\CILA}$ such that, if $\nu(\alpha)=t$, then $\nu(\alpha\wedge\neg\alpha)=T$.

Since $\mathcal{A}_{\CILA}$ is a submultialgebra of $\mathcal{A}_{\mbccl}$ and $\mathcal{F}_{\CILA}$ is a subset of $\mathcal{F}_{\mbccl}$, it is clear that $\mathcal{M}_{\CILA}$ models the axiom schemata and the  inference rule of $\mbccl$. It is easy to prove that it also models \axci, {\bf cf} and \axca$_{\#}$, for $\#\in\{\vee, \wedge, \rightarrow\}$. The following theorem is, again, proved by induction on the length of a derivation.

\begin{theorem}  [Soundness of \cila\ w.r.t. $\mathcal{M}_{\cila}$]
Given a set of formulas $\Gamma\cup\{\varphi\}$ of $\CILA$, if $\Gamma\vdash_{\CILA}\varphi$ then $\Gamma\vDash_{\mathcal{M}_{\CILA}}^\mathsf{RN}\varphi$.
\end{theorem}

\noindent
The proof of completeness of \cila\ w.r.t. $\mathcal{M}_{\cila}$ is similar to that for \mbccl,  now by using the completeness of \cila\ w.r.t. bivaluations. Thus, for any bivaluation $\mathsf{b}$ for $\CILA$, consider the map $\nu:{\bf F}(\Sigma^{\circ},\mathcal{V})\rightarrow\{F,t,T\}$ defined as in the case of $\mbCcl$. Then, the following is obtained:

\begin{theorem}
$\nu$ is a $\Sigma^{\circ}$-homomorphism between ${\bf F}(\Sigma^{\circ}, \mathcal{V})$ and $\mathcal{A}_{\CILA}$ which lies in $\mathcal{F}_{\CILA}$.
\end{theorem}

\begin{theorem} [Completeness of \cila\ w.r.t. $\mathcal{M}_{\cila}$]
	Given a set of formulas $\Gamma\cup\{\varphi\}$ of $\cila$, if $\Gamma\vDash_{\mathcal{M}_{\cila}}^\mathsf{RN}\varphi$ then $\Gamma\vdash_{\cila}\varphi$.
\end{theorem}

\noindent

\subsection{Row-branching truth-tables as a decision procedure for $\cila$} \label{decision-Cila}

The RNmatrix $\mathcal{M}_{\cila}$ induces a simple decision procedure for \cila, while its $\Sigma$-reduct induces one for $C_1$, taking into consideration that \cila\ is a conservative extension of $C_1$. A rigorous proof will be presented now. 

Let $\varphi$ be a formula over the signature $\Sigma^\circ$. A (finite) row-branching truth-table for $\varphi$ can be defined by means of the three-valued multialgebra $\mathcal{A}_{\CILA}$, with the restrictions imposed by $\mathcal{F}_{\cila}$. This is easy to do in a systematic way: let $\varphi_1$, \ldots, $\varphi_k=\varphi$ be the sequence formed by all the subformulas of $\varphi$ linearly ordered by  complexity, that is, $l(\varphi_i) \leq l(\varphi_{i+1})$, for $1 \leq i \leq k-1$, where $l(\alpha)$ denotes the complexity of $\alpha \in {\bf F}(\Sigma^{\circ}, \mathcal{V})$ (formulas with the same complexity are arranged arbitrarily).
Hence, the first $n$ coluns correspond to the propositional variables  $p_1, \ldots, p_n$ occurring in $\varphi$. Given that $\mathcal{A}_{\CILA}$ is a multialgebra, a complex formula $\varphi_i$ can receive more than one truth-value in a row containing  the truth-values of its immediate subformulas; in this case, that row splits into several new ones, one for each possible value assigned to $\varphi_i$ on that row. In order to attend the restriction on valuations of $\mathcal{F}_{\cila}$, it suffices to proceed as follows: if $\varphi_i$ has the form  $\alpha \land \neg\alpha$, then any row in which $\alpha$ (which must appear on a column $\varphi_j$, for $j < i$) gets the value $t$ should split into two rows, one with the value $T$ and the other with the value $t$. Then, the row which assigns the value $t$ to $\varphi_i$  must be discarded. This is illustrated by the following  figure:

{\small
\begin{center}
	\begin{tabular}{| c | c | c | c | c |c |c |c  |c|c|}
		\hline $p_1$ & $\ldots$ & $p_k$  & $\ldots$ &
		$\alpha$ & $\ldots$ & $\neg\alpha$ & $\ldots$ & $\alpha \land \neg\alpha$ & $\ldots$\\
		\hline		
		\multirow{4}{*}{$\nu_0(p_1)$} & \multirow{4}{*}{$\ldots$} & \multirow{4}{*}{$\nu_0(p_k)$}  &
		\multirow{4}{*}{$\ldots$} &  \multirow{4}{*}{$t$} &  \multirow{4}{*}{$\ldots$} &   \multirow{2}{*}{$T$}&  \multirow{2}{*}{$\ldots$} & $T$ & $\ldots$ \\
		\cline{9-10}
		& &  &  &  &  & & &  $\not t$ & discarded\\
		\cline{7-10}
		& &  &  &  &  & \multirow{2}{*}{$t$} &   \multirow{2}{*}{$\ldots$} & $T$ & $\ldots$\\
		\cline{9-10}
		& &  &  &  &  & & &  $\not t$ & discarded\\	
		\hline	
		
	\end{tabular}
\end{center}
}

\

\noindent Then, the process continues until the column corresponding to $\varphi_k=\varphi$ is defined. If $\varphi$ gets a designated value on each (non-discarded) row, it is declared to be valid in \cila; otherwise, it is not.

In order to prove that the process described above constitutes a sound and complete decision procedure for \cila, and so  for $C_1$, a technical result will be stated:

\begin{prop} \label{exte-hom}
Let $\emptyset \neq \Gamma_0 \subseteq {\bf F}(\Sigma^{\circ},\mathcal{V})$  be a finite set closed by subformulas, that is: if $\alpha \in \Gamma_0$ and $\beta$ is a strict subformula of $\alpha$, then $\beta \in \Gamma_0$. Let $\nu_0:\Gamma_0\rightarrow\{T,t,F\}$  be a function satisfying the following:
\begin{enumerate}
\item if $\#\beta$ belongs to $\Gamma_0$, for some  $\# \in \{\neg,\cons\}$, then $\nu_0(\#\beta) \in \tilde{\#} \nu_0(\beta)$; 
\item if $\varphi \, \# \, \psi$  belongs to $\Gamma_0$, for some  $\# \in \{\land, \lor, \to\}$, then $\nu_0(\varphi \, \# \, \psi) \in$\\$ \nu_0(\varphi) \, \tilde{\#} \, \nu_0(\psi)$;
\item if $\alpha \land \neg\alpha$ belongs to $\Gamma_0$ and $\nu_0(\alpha)=t$, then  $\nu_0(\alpha \land \neg\alpha)=T$.
\end{enumerate}
In that case, there exists a homomorphism $\nu$ in $\mathcal{F}_{\cila}$ extending $\nu_0$, i.e., such that $\nu(\alpha)=\nu_0(\alpha)$ for every $\alpha \in \Gamma_0$.
\end{prop} 
\begin{proof}
The valuation $\nu$ will be defined by induction on the complexity $l$ of formulas, which is given by: $l(p)=0$ if $p \in \mathcal{V}$; $l(\neg\alpha)=l(\alpha)+1$; $l(\cons\alpha)=l(\alpha)+2$; and $l(\alpha \,\#\, \beta)=l(\alpha)+l(\beta)+1$.  Moreover, for each  step $n \geq 0$, we must define the value $\nu(\alpha)$, for every formula $\alpha$ with complexity $n$, plus the values $\nu(\neg\alpha)$ and $\nu(\alpha \land \neg\alpha)$.   

Observe that, if $\alpha \notin \Gamma_0$ then, for no formula $\beta$ in $\Gamma_0$, $\alpha$ is a subformula of $\beta$; in particular, $\neg\alpha\notin \Gamma_0$ and $\alpha \land \neg \alpha\notin \Gamma_0$. This means that if $\alpha \not\in \Gamma_0$, the value $\nu(\alpha)$ to be assigned  to $\alpha$, which is chosen with some degree of arbitrariness, will not interfere in the already given value $\nu_0(\beta)$ of any formula $\beta \in \Gamma_0$ in which $\alpha$ is a subformula.
Now, the mapping $\nu$ will be defined inductively.

For every $p \in \mathcal{V}$ define $\nu(p)$ as $\nu_0(p)$, if $p \in \Gamma_0$, and arbitrarily otherwise; this defines $\nu$ for every propositional variable. Now, if $\neg p \in \Gamma_0$, define $\nu(\neg p)=\nu_0(\neg p)$; otherwise, take an arbitrary value in $\tilde{\neg}\nu(p)$.  If $p \land \neg p \in \Gamma_0$, define $\nu(p \land \neg p)=\nu_0(p \land \neg p)$. Otherwise: if $\nu(p)=t$, define $\nu(p\wedge\neg p)=T$; if $\nu(p)\neq t$, define $\nu(p \land \neg p)$ as any value of $\nu(p) \, \tilde{\land} \, \nu(\neg p)$. This concludes the base step $n=0$. 

Suppose that $\nu(\alpha)$, as well as $\nu(\neg\alpha)$ and $\nu(\alpha \land \neg\alpha)$, were already defined, for every formula with complexity $n\geq 0$, satisfying the requirements (induction hypothesis). Let $\alpha$ have complexity $n+1$. Suppose that $\alpha=\neg\beta$. Since $\beta$ has complexity $n$, then  $\nu(\neg\beta)$ was already defined and satisfies $\nu(\neg\beta)=\nu_0(\neg\beta)$, if $\neg\beta \in \Gamma_0$, and  $\nu(\neg\beta) \in \tilde{\neg} \nu(\beta)$ otherwise.  Suppose now that $\alpha=\cons\beta$: if $\alpha \in \Gamma_0$, take $\nu(\alpha)=\nu_0(\alpha)$; otherwise, $\nu(\cons\beta)$ may take any value in $\tilde{\cons} \nu(\beta)$ (by observing that $\nu(\beta)$ was already defined). Now, assume that $\alpha=\beta \, \# \, \gamma$. If $\alpha=\beta \land \neg \beta$ then $\nu(\alpha) \in \nu(\beta)\,\tilde{\land} \, \nu(\neg\beta)$ was already defined and satisfies that $\nu(\beta \land \neg\beta)=\nu_0(\beta \land \neg\beta)$, if $\beta \land \neg\beta \in \Gamma_0$,  and   $\nu(\beta \land \neg\beta)=T$ if $\nu(\beta)=t$. Thus, suppose that $\alpha\neq\beta \land \neg \beta$.  If $\alpha \in \Gamma_0$ then define $\nu(\alpha)=\nu_0(\alpha)$; otherwise, define $\nu(\alpha) \in \nu(\beta) \, \tilde{\#} \, \nu(\gamma)$ arbitrarily, by observing that the values $\nu(\beta)$ and $\nu(\gamma)$ were already defined. Finally, we define the values of $\nu(\neg\alpha)$ and $\nu(\alpha \land \neg\alpha)$. If $\neg\alpha \in \Gamma_0$, take $\nu(\neg\alpha)=\nu_0(\neg\alpha)$; otherwise, $\nu(\neg\alpha)$ takes an arbitrary value in $\tilde{\neg} \nu(\alpha)$. If $\alpha \land \neg\alpha \in \Gamma_0$, define $\nu(\alpha \land \neg\alpha)=\nu_0(\alpha \land \neg\alpha)$; if $\alpha\wedge\neg\alpha$ is not in $\Gamma_0$,  define $\nu(\alpha \land \neg \alpha)=T$ if $\nu(\alpha)=t$ and, otherwise, define $\nu(\alpha \land \neg\alpha) \in \nu(\alpha)\,\tilde{\land} \, \nu(\neg\alpha)$ arbitrarily. This completes the step $n+1$.

It is easy to check that $\nu:{\bf F}(\Sigma^{\circ},\mathcal{V})\rightarrow\{T,t,F\}$  is a function which satisfies the required properties.
\end{proof}

\begin{theorem} \label{decprocC1}
The process described above for constructing a row-branching truth-table for any formula $\varphi$ constitutes a sound and complete decision procedure for \cila\ (hence, for $C_1$) based on the RNmatrix $\mathcal{M}_{\cila}$. That is, a formula $\varphi$ is valid in \cila\ iff  the table defined for $\varphi$ by the process above assigns a designated value to $\varphi$ on each row.
\end{theorem}
\begin{proof}
Given $\varphi$, construct a (necessarily finite) branching truth-table for $\varphi$ as indicated by the process above.  Let $\Gamma_0$ be the set formed by $\varphi$ together with all of its subformulas. Clearly, $\Gamma_0$ satisfies the hypothesis of Proposition~\ref{exte-hom}, and the procedure defines for each row in the table a function $\nu_0:\Gamma_0\rightarrow\{T,t,F\}$ satisfying the hypothesis of Proposition~\ref{exte-hom}. Hence, there exists a homomorphism $\nu$ in $\mathcal{F}_{\cila}$ extending $\nu_0$, i.e., such that $\nu(\alpha)=\nu_0(\alpha)$ for every $\alpha \in \Gamma_0$. Moreover, any 
homomorphism $\nu$ in $\mathcal{F}_{\cila}$ is obtained by extending such mappings $\nu_0$: by restricting $\nu$ to $\Gamma_0$, there is a $\nu_{0}$ whose possible extensions to homomorphisms include $\nu$. Outside $\Gamma_0$, $\nu$ can be defined arbitrarily, while preserving the conditions for being an element of $\mathcal{F}_{\cila}$: the required information for evaluating $\nu(\varphi)$ is already contained in $\Gamma_0$. From these considerations, $\varphi$ is valid in the RNmatrix $\mathcal{M}_{\cila}$ iff  the branching table for $\varphi$ assigns a designated value to $\varphi$ on each row. 
\end{proof}

\

\noindent The results above can be easily adapted to  $C_1$:
 
\begin{defi} \label{defRNC1}
Let $\matGM_{C_1}=(\mathcal{A}_{C_1}, D, \mathcal{F}_{C_1})$ be the RNmatrix obtained from   $\mathcal{M}_{\cila}$ by taking the reduct $\mathcal{A}_{C_1}$ of $\mathcal{A}_{\cila}$ to $\Sigma$ (that is, by `forgetting' $\tilde{\circ}$) and where  $\mathcal{F}_{C_1}$ is the set of valuations $\nu$ over $\mathcal{A}_{C_1}$ such that $\nu(\alpha \land \neg\alpha)=T$ whenever $\nu(\alpha)=t$.
\end{defi}

\begin{theorem} [Soundness and completeness of $C_1$ w.r.t.  $\matGM_{C_1}$] \label{souncompRNC1}
Let $\Gamma \cup \{\varphi\} \subseteq \bF(\Sigma,\matV)$. Then: $\Gamma \vdash_{C_1} \varphi$ \ iff \ $\Gamma \vDash_{\matGM_{C_1}}^\mathsf{RN} \varphi$. 
\end{theorem}
\begin{proof}
It is an immediate consequence of the corresponding result for \cila\ w.r.t. $\mathcal{M}_{\cila}$, by definition of $\matGM_{C_1}$, and by the fact that \cila\ is a conservative extension of $C_1$.
\end{proof}

\begin{remark} \label{decprocC1mbcCl}
It is worth noting that the RNmatrix $\matGM_{C_1}$ defines a decision procedure for da Costa's logic $C_1$.
By similar considerations, it is easy to prove that the RNmatrix  $\mathcal{M}_{\mbCcl}$ for \mbccl\ constitutes a decision procedure for \mbccl. This solves the decidability of such logics by means of three-valued RNmatrices, recalling that such logics cannot be characterized by mere finite Nmatrices. It is worth noting that in~\cite[Corollary~8.106]{Avron:Arieli:Zamansky:18} it is presented a decision procedure for \cila/$C_1$ which can be obtained from the infinite characteristic  Nmatrix for \cila\ introduced in~\cite[Example~8.99]{Avron:Arieli:Zamansky:18}. As we shall see in Example~\ref{ex3valC1}, this procedure decision is much more complicated than the one introduced here by means of the three-valued RNmatrix  $\matGM_{C_1}$.
\end{remark} 

\begin{example} \label{ex3valC1}
Let $p$ be a propositional variable,  $\varphi=(p\wedge\neg p) \wedge \neg(p\wedge\neg p)$ and $\psi=\varphi \to \neg\neg p$. The following branching table in $\matGM_{C_1}$, constructed according to the procedure described at the beginning of this subsection, shows that $\varphi$ is unsatisfiable in \cila/$C_1$, hence $\psi$ is valid.

\begin{center}
\begin{tabular}{|p{1cm}|p{1cm}|p{1cm}|p{1cm}|p{1.5cm}|p{1cm}|p{1cm}|}
\hline
$p$ & $\neg p$ & $p\wedge\neg p$ & $\neg\neg p$ & $\neg(p\wedge\neg p)$  & $\varphi$ & $\psi$\\\hline
$T$ & $F$ & $F$ & $T$ & $T$ & $F$ & $T$\\\hline
\multirow{4}{*}{$t$} & $T$ & $T$ & $F$ & $F$ & $F$ & $T$\\\cline{2-7}
& \multirow{3}{*}{$t$} & \multirow{3}{*}{$T$} & $T$ & $F$ & $F$ & $T$\\\cline{4-7}
& & & \multirow{2}{*}{$t$} & \multirow{2}{*}{$F$} & \multirow{2}{*}{$F$} & $T$\\ \cline{7-7}
&&&&&& $t$\\\hline
$F$ & $T$ & $F$ & $F$ & $T$ & $F$ & $T$\\\hline
\end{tabular}
\end{center}

Now, in~\cite[Example~8.99]{Avron:Arieli:Zamansky:18} it was presented an infinite characteristic NMatrix for \cila/$C_1$ with domain $\{f\}\cup \bigcup \{\{t^j_i,\top^j_i\} \ : \ i,j \geq 0\}$. By  the decision procedure for \cila/$C_1$ described in the proof of~\cite[Corollary~8.106]{Avron:Arieli:Zamansky:18} it follows that the validity of $\psi$ must be analyzed within a sub-Nmatrix of the infinite one by restricting the truth-values to the set $\{f\}\cup \bigcup \{\{t^j_i,\top^j_i\} \ : \ 0 \leq i \leq n^*(\psi), \, 0 \leq j \leq k^*(\psi)\}$. Here, $n^*(\psi)=4$ is the number of subformulas of $\psi$ which do not begin with $\neg$, while $k^*(\psi)=2$ is the maximum number of consecutive negation symbols $\neg$ occurring in $\psi$. This means that, in order to analyze  the validity of $\psi$ by means of this decision procedure, it is required to check its validity whithin an Nmatrix with $1 + 2 \times 3 \times 5=31$ truth-values. Obviously this process is much more expensive than the six-rows branching table generated by  $\matGM_{C_1}$ displayed above. Clearly, the difference between the complexity of both decision procedures increases as the complexity of $\psi$ increases. 
\end{example}

\section{Finite RNmatrices for da Costa's hierarchy} \label{SectCn}
\vspace*{-3mm}
\subsection{An historical antecedent: Fidel structures} \label{FCn}

In 1977 Fidel proved, for the first time, the decidability of da Costa's calculi $C_n$ by means of an original algebraic-relational class of structures now called {\em Fidel structures} (see~\cite{fid:77}).\footnote{As mentioned right before Theorem~\ref{comple-Cn-biv}, da Costa and Alves,  also in 1977, introduced the method of valuations as a decision procedure for the calculi $C_n$. However, the first correct proof of soundness and completeness of that method appeared in 1980, see~\cite{lop:alv:80}. It is interesting to note that Fidel's paper was submitted on September 14, 1976 while da Costa and Alves's one was submitted on October 12, 1976, practically simultaneously.} A Fidel structure for $C_n$ is a triple $\mathcal{N}=(\mathcal{A},\{N_a\}_{a \in A},\{N^{(n)}_a\}_{a \in A})$ such that $\mathcal{A}$ is a Boolean algebra with universe $A$ and,  for every $a\in A$, $N_a,N^{(n)}_a \subseteq A$ satisfy certain properties (see~\cite[Definition~1, pp. 32-33]{fid:77}). A {\em valuation over $\mathcal{N}$ for $C_n$} is a function $v:{\bf F}(\Sigma, \mathcal{V}) \to A$ satisfying, among other  properties, that $v(\alpha \# \beta) = v(\alpha) \# v(\beta)$ for every $\# \in \Sigma_2$, $v(\neg\alpha) \in N_{v(\alpha)}$  and $v(\alpha^{(n)}) \in N^{(n)}_{v(\alpha)}$ (see~\cite[Definition~1, pp. 37-38]{fid:77}). Let $\tilde{\neg}:A \to \wp(A)\setminus\{\emptyset\}$ and $\tilde{(n)}:A \to \wp(A)\setminus\{\emptyset\}$ be given by $\tilde{\neg}\,a:=N_a$ and $\tilde{(n)}\,a:=N^{(n)}_a$, respectively. It is easy to see that the consequence relation induced by the Fidel structure $\mathcal{N}$ can be described as a structural RNmatrix $\mathcal{M}_\mathcal{N}^n=(\mathcal{A}^+, D, \mathcal{F}_\mathcal{N}^n)$ such that $\mathcal{A}^+$ is the expansion of $\mathcal{A}$ (seen as a multialgebra) by adding the multioperators $\tilde{\neg}$ and $\tilde{(n)}$, $D=\{1\}$ and $\mathcal{F}_\mathcal{N}^n$ is the set of valuations over $\mathcal{N}$ for $C_n$. Hence, Fidel structures semantics for $C_n$ are an early example of (structural) RNmatrix semantics. Analogously, it can be proven that the Fidel structures semantics considered in the literature for other logics (for instance, for several {\bf LFI}s including \mbc, \mbccl\ and \cila, as proposed in~\cite[Chapter~6]{CC16}) can be recast as (structural) RNmatrix semantics. 

In~\cite[p. 34]{fid:77} a Fidel structure over the two-element Boolean algebra {\bf 2}, called  {\bf C}, is defined, which induces a decision procedure  for $C_n$ (\cite[Theorem~5(vi)]{fid:77}). The RNmatrix $\mathcal{M}_{\bf C}^n$ for $C_n$ associated to  {\bf C} is based on the expansion $ {\bf 2}^+$ of  {\bf 2} (seen as a multialgebra) with the multioperators defined by $\tilde{\neg}\,0=\tilde{(n)}\,0=\{1\}$ and $\tilde{\neg}\,1=\tilde{(n)}\,1=\{0,1\}$. Together with this Nmatrix (where  $D=\{1\}$) it is considered a rather complicated set of valuations  $\mathcal{F}_{\bf C}^n$ for $C_n$ over {\bf C} (see~\cite[Definition~1, pp. 37-38]{fid:77}). It is interesting to observe that Fidel's classical result states that every calculus $C_n$ is decidable by a single two-element RNmatrix, to the cost of considering a set of valuations defined in a complex way.\footnote{By Remark~\ref{bival-RN}, it should be clear that Fidel's  and da Costa-Alves-Lopari\'c's solutions to the decidability problem of da Costa's systems are of a similar nature.}
One of the main goals of this paper is presenting an alternative solution of this problem by means of a more intuitive and easy to deal finite RNmatrix for da Costa's hierarchy of paraconsistent systems. This task will be accomplished in the next subsections.

\subsection{RNmatrices for $C_n$, with $n\geq 2$} \label{sect-Cn}

In Subsection~\ref{decision-Cila}  it was proved that the  RNmatrix $\matGM_{C_1}$ introduced in Definition~\ref{defRNC1} gives origin to a simple decision procedure for da Costa's $C_1$.
Now, the general case of $C_n$, for $n \geq 2$, will be analyzed.  Recall from Section~\ref{Cn} the signature $\Sigma$, as well as the notation $\alpha^1=\alpha^\cons= \lnot (\alpha \land \lnot \alpha)$. Hence, $\alpha^2=\alpha^{\cons\cons}= \lnot (\alpha^\cons \land \lnot (\alpha^\cons))$. In general, $\alpha^0= \alpha$ and $\alpha^{j+1}= \lnot (\alpha^j \land \lnot (\alpha^j))$, for $j \geq 0$. On the other hand, $\alpha^{(j)}=\alpha^1 \land\ldots \land \alpha^j$. 

To start with, let us consider $(n+1)$-tuples $z=(z_1,z_2,\ldots,z_{n+1})$ in ${\bf 2}^{n+1}$ such that each coordinate is given by $\mathsf{b}(\alpha)$, $\mathsf{b}(\neg\alpha)$, $\mathsf{b}(\alpha^1)$, $\mathsf{b}(\alpha^2)$, \ldots, $\mathsf{b}(\alpha^{n-1})$, respectively, for a given  $C_n$-bivaluation $\mathsf{b}$ and a formula $\alpha$. From the basic properties of $\mathsf{b}$ stated in Definition~\ref{bival-def}, it follows that there are exactly $n+2$ of such tuples, namely:
\[T_n=(1,0,1,\ldots,1),\quad t^n_0=(1,1,0,1,\ldots,1)\quad\cdots\quad t^n_{n-2}=(1,1,\ldots,1,0),\]
\[t^n_{n-1}=(1,1,\ldots,1)\quad\text{and}\quad F_n=(0,1,1,\ldots,1).\]
Each of such tuples will be called {\em snapshots} for $C_n$, by adopting the terminology introduced in~\cite[Chapter~6]{CC16} in the context of {\em swap structures} (which are multialgebras of snapshots).  Notice that an element of $\textbf{2}^{n+1}$ is a snapshot iff it contains at most one coordinate equal to $0$. Alternatively, the set of snapshots for $C_n$ may be defined without using $\mathsf{b}$ and $\alpha$ as $$B_n=\{z \in {\bf 2}^{n+1} \ : \ \big(\bigwedge_{i=1}^k z_i\big) \lor z_{k+1} = 1 \ \mbox{ for every } \ 1 \leq k \leq n\}.$$ 

\begin{defi} \label{def-sets}
Consider the following relevant subsets of $B_n$:
\begin{itemize}
	\item[-] $D_n:=B_n \setminus\{F_n\} = \{z \in B_n \ : \ z_1=1\}$ ({\em designated values});
	\item[-] $U_n:=\{F_n\}=B_n \setminus D_n$ ({\em undesignated values});
	\item[-] $Boo_n:=\{T_n,F_n\} = \{z \in B_n \ : \ z_1 \land z_2=0\}$ ({\em Boolean values});
	\item[-] $I_n:=B_n \setminus Boo_n = B_n \setminus\{T_n,F_n\}$ ({\em inconsistent values}).
\end{itemize}
\end{defi}

\noindent
Notice that $z \in Boo_n$ iff $z=(a,\sneg a, 1, \ldots, 1)$ for some $a \in {\bf 2}$ (here, $\sneg a$ denotes the Boolean complement of $a$ in {\bf 2}).
Now, a  $(n+2)$-valued multialgebra over $\Sigma$ with domain $B_n$ called $\A_{C_n}$ will be defined as being a swap structure based on the restrictions imposed by Definition~\ref{bival-def}:\\

\begin{defi} \label{defACn} The multialgebra $\A_{C_n} = ( B_n, \tilde{\land}, \tilde{\lor}, \tilde{\to}, \tilde{\neg})$ over $\Sigma$ is defined as follows, for any $z,w \in B_n$:\\

$\begin{array}{lccl}
(C^n_{\tilde{\neg}}) & \tilde{\neg}\, z &=& \{w \in B_n  \ : \ w_1 = z_2 \ \mbox{ and } \ w_2 \leq z_1\}\\[4mm]
(C^n_{\tilde{\#}}) & z \,\tilde{\#}\, w &=& \left \{ \begin{tabular}{ll}
$\{u \in Boo_n  \ : \ u_1 = z_1 \# w_1\}$ & if $z,w\in Boo_n$,\\[3mm]
$\{u \in B_n  \ : \ u_1 = z_1 \# w_1\}$ & otherwise,\\
\end{tabular}\right.  \\[8mm]
&&&\mbox{ for } \# \in \{\land, \lor, \to\}.
\end{array}$
\end{defi}

\

\noindent Observe that, if $z,w\in Boo_n$ then $z \,\tilde{\#}\, w$ is the singleton $\{T_n\}$ or $\{F_n\}$, which is calculated from the two-element truth-tables of \cpl\ (as it was done with $C_1$). In addition, if $z \in Boo_n$, then $z=(a,\sneg a, 1, \ldots, 1)$ for some $a \in {\bf 2}$, and so $\tilde{\neg} \, z = \{(\sneg a,a, 1, \ldots, 1)\} \subseteq Boo_n$. The multioperations of the multialgebra $\A_{C_n}$ can be presented in  a compact form as follows:

\

\begin{center}
	\begin{tabular}{| c | c | }
		\hline $z$ & $\tilde{\neg}\, z$  \\
		\hline
		
		$T_n$ & $F_n$\\
		\hline
		
		$t^n_i$ &  $D_n$\\
		\hline
		
		$F_n$ &  $T_n$\\
		\hline
		
	\end{tabular}
	\hspace{3.2cm}
	\begin{tabular}{| c | c | c | c |}
		\hline $\tilde{\to}$ & $T_n$ & $t^n_j$ & $F_n$ \\
		\hline
		
		$T_n$ & $T_n$ &  $D_n$ &  $F_n$\\
		\hline
		
		$t^n_i$ &  $D_n$  & $D_n$  & $F_n$\\
		\hline
		
		$F_n$ &  $T_n$ & $D_n$  & $T_n$\\
		\hline
		
	\end{tabular}
\end{center}

\begin{center}
	\begin{tabular}{| c | c | c | c |}
		\hline $\tilde{\land}$ & $T_n$ & $t^n_j$ & $F_n$ \\
		\hline
		
		$T_n$ & $T_n$ &  $D_n$  &  $F_n$\\
		\hline
		
		$t^n_i$ &  $D_n$  & $D_n$  & $F_n$\\
		\hline		
		
		$F_n$ &  $F_n$  & $F_n$  & $F_n$\\
		\hline
		
	\end{tabular}
	\hspace{1cm}
	\begin{tabular}{| c | c | c | c |}
		\hline $\tilde{\lor}$ & $T_n$ & $t^n_j$ & $F_n$ \\
		\hline
		
		$T_n$ & $T_n$ &  $D_n$  &  $T_n$\\
		\hline
		
		$t^n_i$ &  $D_n$  & $D_n$  & $D_n$\\
		\hline
		
		$F_n$ &  $T_n$  & $D_n$  & $F_n$\\
		\hline
		
	\end{tabular}
\end{center}

\

\noindent Recall the set $I_n$ of inconsistent values introduced in Definition~\ref{def-sets}. Now, as it was done with $C_1$, a suitable set of valuations over $\A_{C_n}$ will be considered:

\begin{defi} \label{valCn} Let $\matF_{C_n}$ be the set of valuations $\nu$ over $\A_{C_n}$ (that is, homomorphisms of multialgebras $v:\bF(\Sigma,\matV) \to \A_{C_n}$) such that, for every $\alpha$:\\
	
	$\begin{array}{cl}
	(1) & \nu(\alpha) = t^n_0 \ \mbox{ implies that } \ \nu(\alpha \land \neg\alpha)=T_n;\\[1mm]
	(2) & \nu(\alpha) = t^n_{k} \ \mbox{ implies that } \ \nu(\alpha \land \neg\alpha) \in I_n$ and $\nu(\alpha^1) = t^n_{k-1},   \\[1mm]
	& \mbox{for every } \ 1\leq k\leq n-1;\\[3mm]
	\end{array}$
	
	\noindent	Let	$\matGM_{C_n} = ( \A_{C_n}, D_n,\matF_{C_n} )$ be the restricted Nmatrix obtained from this.
\end{defi}

\noindent Clearly,  $\matGM_{C_n}$ is structural. It is worth noting that, for every $\nu \in  \matF_{C_n}$ and formula $\alpha$, $\nu(\alpha) = t^n_0$ implies that $\nu(\alpha^1) = F_n$. Hence, if $\nu(\alpha) = t^n_1$, then $\nu(\alpha^2) = F_n$. In general, if $\nu(\alpha) = t^n_i$, then $\nu(\alpha^{i+1}) = F_n$ and $\nu(\alpha^{j}) \in D_n$ for every $0\leq j\leq i$ and formula $\alpha$.  This produces the following scenario in $\matGM_{C_n}$, where $X^*$ means that the value $X$ is chosen by a valuation in $\matF_{C_n}$:\vspace*{-5mm}
\begin{center}

	\captionof{table}{}\label{tabela3}
	\begin{tabular}{| c | c | c | c | c | c | c | c | c | c |}
		\hline {\scriptsize$\alpha$} &{\scriptsize$\alpha \land \lnot \alpha$}&{\scriptsize$\alpha^1$}&\hspace{-1mm}{\scriptsize$\alpha^1 \land \lnot \alpha^1$}\hspace{-1mm}&{\scriptsize$\alpha^2$}& \ldots &\hspace{-1mm}{\scriptsize$\alpha^{n-1}$}\hspace{-1mm}&{\scriptsize$\alpha^{n-1} \land \neg \alpha^{n-1}$}&{\scriptsize$\alpha^{n}$}&{\scriptsize$\alpha^{(n)}$} \\
		\hline
		
		$T_n$ & $F_n$ & $T_n$ & $F_n$ & $T_n$ & \ldots & $T_n$ & $F_n$  & $T_n$  &  $T_n$\\
		\hline
		
		$t^n_0$ &  $T_n^*$ & $F_n$ & $F_n$ & $T_n$  & \ldots & $T_n$ & $F_n$ & $T_n$ & $F_n$\\
		\hline
		
		$t^n_1$ & $I_n^*$ & $t^{n *}_0$ & $T_n^*$ & $F_n$ &  \ldots & $T_n$ & $F_n$  & $T_n$ & $F_n$ \\	
		\hline
		
		$t^n_2$ & $I_n^*$ & $t^{n *}_1$ & $I_n^*$ & $t^{n *}_0$ &  \ldots & $T_n$ & $F_n$  & $T_n$ & $F_n$ \\	
		\hline
		
		\vdots&\vdots&\vdots&\vdots&\vdots&$\ddots$&\vdots&\vdots&\vdots&\vdots\\	
		\hline
		
		$t^n_{n-3}$ & $I_n^*$ & $t^{n *}_{n-4}$ & $I_n^*$ & $t^{n *}_{n-5}$ &  \ldots & $T_n$ & $F_n$ & $T_n$  & $F_n$\\
		\hline
		
		$t^n_{n-2}$ & $I_n^*$ & $t^{n *}_{n-3}$ & $I_n^*$ & $t^{n *}_{n-4}$ &  \ldots & $F_n$ & $F_n$ & $T_n$  & $F_n$\\
		\hline
				
		$t^n_{n-1}$ & $I_n^*$ & $t^{n *}_{n-2}$ & $I_n^*$ & $t^{n *}_{n-3}$ &  \ldots & $t^{n *}_0$ & $T_n^*$ & $F_n$  & $F_n$\\
		\hline
		
		$F_n$ &  $F_n$  & $T_n$ & $F_n$ & $T_n$  &  \ldots & $T_n$ & $F_n$  & $T_n$ & $T_n$ \\ 
		\hline
		
	\end{tabular}

\end{center}

\

\noindent Observe that the restrictions imposed by a valuation $\nu$ in $\matF_{C_n}$ to the values of formulas of the form $\alpha^j \land \neg(\alpha^j)$ and $\alpha^k$, when $\nu(\alpha)=t^n_i$, increase when $i$ increases. Namely, if  $\nu(\alpha)=t^n_i$, then $\nu$ must restrict the values of  $2i+1$ formulas involving $\alpha$:  $\alpha \land \neg\alpha$, $\alpha^j$  and $\alpha^j \land \neg(\alpha^j)$ for $1 \leq j \leq i$.

For any formula $\alpha$ let  $\sneg \alpha: = \neg \alpha \land  \alpha^{(n)}$ be the strong negation definable in $C_n$ (see~\cite{dC63}). Then, for any $\nu \in \matF_{C_n}$ we have that  $\nu(\sneg\alpha)=F_n$ if $\nu(\alpha) \in D_n$, and $\nu(\sneg\alpha)=T_n$ otherwise. Hence, if $\bot:=\alpha \land \sneg \alpha$ and $\top:=\alpha \lor \sneg\alpha$, then $\nu(\bot)=F_n$ and $\nu(\top) \in D_n$.

The proof of soundness and completeness of $C_n$ w.r.t. $\matGM_{C_n}$ requires the following three technical lemmas. The first one is easily proved with induction over $k$, and it reflects the content of Table~\ref{tabela3}. 

\begin{lemma}\label{lem_tech}
If  $\nu \in\mathcal{F}_{C_{n}}$ and $1\leq k\leq n$ then, for any $\alpha$:
\begin{enumerate}
\item if $\nu(\alpha)=T_{n}$, then $\nu(\alpha^{k})=T_{n}$;
\item if $\nu(\alpha)=t^{n}_{i}$, for some $0\leq i\leq k-2$, then $\nu(\alpha^{k})=T_{n}$;
\item if $\nu(\alpha)=t^{n}_{k-1}$, then $\nu(\alpha^{k})=F_{n}$;
\item if $\nu(\alpha)=t^{n}_{i}$, for $k\leq i\leq n-1$, then $\nu(\alpha^{k})=t^{n}_{i-k}$;
\item if $\nu(\alpha)=F_{n}$, then $\nu(\alpha^{k})=T_{n}$.
\end{enumerate}
\end{lemma}

\begin{lemma} \label{lem-sound-Cn}
	Let $\nu$ be a valuation in $\matF_{C_n}$ and write, for any formula $\alpha$, $\nu(\alpha)=(\nu(\alpha)_{1}, \nu(\alpha)_{2},\ldots , \nu(\alpha)_{n+1})$. Then, the mapping $\mathsf{b}:\bF(\Sigma,\matV) \to {\bf 2}$ given by $\mathsf{b}(\alpha):=\nu(\alpha)_1$ is a $C_n$-bivaluation such that $\mathsf{b}(\alpha)=1$ iff $\nu(\alpha) \in D_n$. 
\end{lemma}

\begin{proof}
For any $\#\in\{\vee, \wedge, \rightarrow\}$, $\mathsf{b}(\alpha\#\beta)=1$ if and only if $\nu(\alpha\#\beta)_{1}=1$; since $\nu(\alpha\#\beta)_{1}=\nu(\alpha)_{1}\#\nu(\beta)_{1}$ by the definition of $\tilde{\#}$, we have that:
\begin{enumerate}
\item $\mathsf{b}(\alpha\vee\beta)=1$ if and only if either $\mathsf{b}(\alpha)=1$ or $\mathsf{b}(\beta)=1$ (clause $(B2)$ for being a $C_{n}$-bivaluation);
\item $\mathsf{b}(\alpha\wedge\beta)=1$ if and only if $\mathsf{b}(\alpha)=1$ and $\mathsf{b}(\beta)=1$ (clause $(B1)$);
\item $\mathsf{b}(\alpha\rightarrow\beta)=1$ if and only if $\mathsf{b}(\alpha)=0$ or $\mathsf{b}(\beta)=1$ (clause $(B3)$).
\end{enumerate}
If $\mathsf{b}(\alpha)=0$, then $\nu(\alpha)_{1}=0$; since $\nu(\alpha)_{1}\vee\nu(\alpha)_{2}=1$, by definition of $B_{n}$, and $\nu(\neg\alpha)_{1}=\nu(\alpha)_{2}$, by definition of $\tilde{\neg}$, we find that $\mathsf{b}(\neg\alpha)=\nu(\neg\alpha)_{1}=1$, satisfying clause $(B4)$. If $\mathsf{b}(\neg\neg\alpha)=1$, then $\nu(\neg\neg\alpha)_{1}=1$. But $\nu(\neg\neg\alpha)_{1}=\nu(\neg\alpha)_{2}$ and $\nu(\neg\alpha)_{2}\leq \nu(\alpha)_{1}$ since $\nu(\neg\alpha) \in \tilde{\neg}\nu(\alpha)$. Hence $1 \leq \nu(\alpha)_{1}$, that is, $\mathsf{b}(\alpha)=\nu(\alpha)_{1}=1$ and so $(B5)$ is satisfied.
If $\mathsf{b}(\alpha^{n-1})=\mathsf{b}(\neg(\alpha^{n-1}))$, then $\nu(\alpha^{n-1})_{1}=\nu(\neg(\alpha^{n-1}))_{1}=\nu(\alpha^{n-1})_{2}$. By definition of $B_{n}$,  $\nu(\alpha^{n-1})\in I_{n}$. From Lemma~\ref{lem_tech}, $\nu(\alpha)=t^{n}_{n-1}$ and $\nu(\alpha^{n})=F_{n}$, hence $\mathsf{b}(\alpha^{n})=\nu(\alpha^{n})_{1}=0$. Conversely, if $\mathsf{b}(\alpha^{n})=0$, $\nu(\alpha^{n})=F_{n}$ and, from Lemma \ref{lem_tech}, $\nu(\alpha)=t^{n}_{n-1}$. So $\nu(\alpha^{n-1})=t^{n}_{0}$, $\nu(\neg(\alpha^{n-1}))\in D_{n}$ and therefore $\mathsf{b}(\alpha^{n-1})=\nu(\alpha^{n-1})_{1}=1=\nu(\neg(\alpha^{n-1}))_{1}=\mathsf{b}(\neg(\alpha^{n-1}))$.
Thus, $(B6)_{n}$ is satisfied. If $\mathsf{b}(\alpha)=\mathsf{b}(\neg\alpha)$ then $\nu(\alpha)_{1}=\nu(\neg\alpha)_{1}=\nu(\alpha)_{2}=1$, so $\nu(\alpha)\in I_{n}$. By Lemma \ref{lem_tech},  $\nu(\alpha^{1})\neq T_{n}$ and so $\mathsf{b}(\neg(\alpha^{1}))=\nu(\neg(\alpha^{1}))_{1}=\nu(\alpha^{1})_{2}=1$. Conversely, if $\mathsf{b}(\neg(\alpha^{1}))=1$ then $\nu(\neg(\alpha^{1}))_{1}=\nu(\alpha^{1})_{2}=1$, hence $\nu(\alpha^{1}) \neq T_{n}$. From this, $\nu(\alpha)\in I_{n}$ and so  $\nu(\neg\alpha)\in D_{n}$. Thus, $\mathsf{b}(\alpha)=\nu(\alpha)_{1}=1$ and $\mathsf{b}(\neg\alpha)=\nu(\neg\alpha)_{1}=1$ and $(B7)$ holds.
If $\mathsf{b}(\alpha)\neq\mathsf{b}(\neg\alpha)$ and $\mathsf{b}(\beta)\neq\mathsf{b}(\neg\beta)$, $\nu(\alpha)_{1}\neq\nu(\neg\alpha)_{1}=\nu(\alpha)_{2}$ and $\nu(\beta)_{1}\neq\nu(\neg\beta)_{1}=\nu(\beta)_{2}$, meaning $\nu(\alpha), \nu(\beta)\in \{F_{n}, T_{n}\}$. From the tables for $\#\in\{\vee, \wedge, \rightarrow\}$ we see that $\nu(\alpha\#\beta)\in\{F_{n}, T_{n}\}$, and so $\mathsf{b}(\alpha\#\beta)\neq \mathsf{b}(\neg(\alpha\#\beta))$. Hence clause $(B8)$ holds, which finishes the proof.
\end{proof}

\begin{lemma} \label{lem-compl-Cn}
	Let $\mathsf{b}$ be a $C_n$-bivaluation. Then, the mapping $\nu:\bF(\Sigma,\matV) \to B_n$ given by $\nu(\alpha):=(\mathsf{b}(\alpha),\mathsf{b}(\neg\alpha),\mathsf{b}(\alpha^1),\ldots,\mathsf{b}(\alpha^{n-1}))$  is  a valuation in $\matF_{C_n}$ such that, for every formula $\alpha$, $\mathsf{b}(\alpha)=1$ iff $\nu(\alpha) \in D_n$.
\end{lemma}

\begin{proof} Clearly $\nu(\alpha) \in B_n$ for every $\alpha$.
First, we show $\nu$ is a homomorphism. We have that $\nu(\neg\alpha)=(\mathsf{b}(\neg\alpha), \mathsf{b}(\neg\neg\alpha), \ldots)$, hence: if $\nu(\alpha)=T_{n}$, $\mathsf{b}(\neg\alpha)=0$ and therefore $\nu(\neg\alpha)=F_{n}$; if $\nu(\alpha)\in I_{n}$, $\mathsf{b}(\neg\alpha)=1$, implying $\nu(\neg\alpha)\in D_{n}$; finally, if $\nu(\alpha)=F_{n}$, $\mathsf{b}(\alpha)=0$ and $\mathsf{b}(\neg\alpha)=1$ and from $(B5)$, $\mathsf{b}(\neg\neg\alpha)=0$, hence $\nu(\neg\alpha)=T_{n}$. Thus, $\nu(\neg\alpha)\in \tilde{\neg}\nu(\alpha)$.

Note that $\nu(\alpha\vee\beta)=(\mathsf{b}(\alpha\vee\beta), \mathsf{b}(\neg(\alpha\vee\beta)), \ldots )$. If either $\nu(\alpha)$ or $\nu(\beta)$ equals $T_{n}$ and both are Boolean-valued, either $\mathsf{b}(\alpha)$ or $\mathsf{b}(\beta)$ equals $1$. Hence, $\mathsf{b}(\alpha\vee\beta)=1$ by $(B2)$. Since $\mathsf{b}(\alpha)\neq\mathsf{b}(\neg\alpha)$ and $\mathsf{b}(\beta)\neq\mathsf{b}(\neg\beta)$ then, by $(B8)$, $\mathsf{b}(\alpha\vee\beta)\neq\mathsf{b}(\neg(\alpha\vee\beta))$. Then, $\nu(\alpha\vee\beta)=T_{n}$. If $\nu(\alpha)$ or $\nu(\beta)$ is in $I_{n}$, we have that either $\mathsf{b}(\alpha)=1$ or $\mathsf{b}(\beta)=1$ and, by $(B2)$, $\mathsf{b}(\alpha\vee\beta)=1$, hence $\nu(\alpha\vee\beta)\in D_{n}$. If $\nu(\alpha)=\nu(\beta)=F_{n}$, $\mathsf{b}(\alpha)=\mathsf{b}(\beta)=0$ and so $\mathsf{b}(\alpha\vee\beta)=0$, by $(B2)$. Thus, $\nu(\alpha\vee\beta)=F_{n}$. In all cases, $\nu(\alpha\vee\beta) \in \nu(\alpha)\,\tilde{\vee}\,\nu(\beta)$.

By definition, $\nu(\alpha\wedge\beta)=(\mathsf{b}(\alpha\wedge\beta), \mathsf{b}(\neg(\alpha\wedge\beta)), \ldots )$. If $\nu(\alpha)=\nu(\beta)=T_{n}$, $\mathsf{b}(\alpha)=\mathsf{b}(\beta)=1$ and $\mathsf{b}(\neg\alpha)=\mathsf{b}(\neg\beta)=0$. By $(B1)$, $\mathsf{b}(\alpha\wedge\beta)=1$, and by $(B8)$, $\mathsf{b}(\neg(\alpha\wedge\beta))=0$. Hence $\nu(\alpha\wedge\beta)=T_{n}$. If $\nu(\alpha)$ or $\nu(\beta)$ equals $F_{n}$, either $\mathsf{b}(\alpha)$ or $\mathsf{b}(\beta)$ equals $0$, and so $\mathsf{b}(\alpha\wedge\beta)=0$, by $(B1)$. Hence $\nu(\alpha\wedge\beta)=F_{n}$. In the remaining cases, when either $\nu(\alpha)$ or $\nu(\beta)$ is in $I_{n}$ and both are designated, one sees that $\mathsf{b}(\alpha)=\mathsf{b}(\beta)=1$ and therefore $\mathsf{b}(\alpha\wedge\beta)=1$. Hence $\nu(\alpha\wedge\beta)\in D_{n}$. In all cases, $\nu(\alpha\wedge\beta)\in \nu(\alpha)\,\tilde{\wedge}\,\nu(\beta)$.

Note that $\nu(\alpha\rightarrow\beta)=(\mathsf{b}(\alpha\rightarrow\beta), \mathsf{b}(\neg(\alpha\rightarrow\beta)), \ldots)$. If $\nu(\alpha)=F_{n}$ or $\nu(\beta)=T_{n}$, and both are Boolean-valued, $\mathsf{b}(\alpha)=0$ or $\mathsf{b}(\beta)=1$, and both $\mathsf{b}(\alpha)\neq\mathsf{b}(\neg\alpha)$ and $\mathsf{b}(\beta)\neq\mathsf{b}(\neg\beta)$. By  $(B3)$ and $(B8)$,  $\mathsf{b}(\alpha\rightarrow\beta)=1$ and $\mathsf{b}(\neg(\alpha\rightarrow\beta))=0$, and so $\nu(\alpha\rightarrow\beta)=T_{n}$. If $\nu(\beta)\in I_{n}$, $\mathsf{b}(\beta)=1$ and so $\mathsf{b}(\alpha\rightarrow\beta)=1$, by $(B3)$. Hence $\nu(\alpha\rightarrow\beta)\in D_{n}$. If $\nu(\beta)=F_{n}$ and $\nu(\alpha)\in D_{n}$, $\mathsf{b}(\beta)=0$ and $\mathsf{b}(\alpha)=1$, hence $\mathsf{b}(\alpha\rightarrow\beta)=0$, by $(B3)$. Thus, $\nu(\alpha\rightarrow\beta)=F_{n}$. Finally, if $\nu(\alpha)\in I_{n}$ and $\nu(\beta)=T_{n}$, $\mathsf{b}(\alpha)=1$ and $\mathsf{b}(\beta)=1$. By $(B3)$, $\mathsf{b}(\alpha\rightarrow\beta)=1$ and therefore $\nu(\alpha\rightarrow\beta)\in D_{n}$. In all cases, $\nu(\alpha\rightarrow\beta)\in\nu(\alpha)\,\tilde{\rightarrow}\,\nu(\beta)$.

It remains to prove that $\nu$ is in $\mathcal{F}_{C_{n}}$. If $\nu(\alpha)=t^{n}_{0}$, we have $\mathsf{b}(\alpha)=\mathsf{b}(\neg\alpha)=1$ (and so $\mathsf{b}(\alpha\wedge\neg\alpha)=1$) and $\mathsf{b}(\neg(\alpha\wedge\neg\alpha))=\mathsf{b}(\alpha^{1})=0$, hence $\nu(\alpha\wedge\neg\alpha)=T_{n}$. If $\nu(\alpha)=t^{n}_{k}$, for $1\leq k\leq n-1$, we have $\mathsf{b}(\alpha)=\mathsf{b}(\neg\alpha)=1$ (meaning $\mathsf{b}(\alpha\wedge\neg\alpha)=1$) and $\mathsf{b}(\neg(\alpha\wedge\neg\alpha))=\mathsf{b}(\alpha^{1})=1$, and so $\nu(\alpha\wedge\neg\alpha)\in I_{n}$. Furthermore, $\mathsf{b}((\alpha^{1})^{k})=\mathsf{b}(\alpha^{k+1})=0$, from the fact that $\nu(\alpha)=t^{n}_{k}$. But $\mathsf{b}((\alpha^{1})^{k})=0$ implies that $\nu(\alpha^{1})=t^{n}_{k-1}$, which ends the proof.
\end{proof}

From this, the soundness and completeness of $C_n$ with respect to the finite RNmatrix  $\matGM_{C_n}$ is easily obtained.

\begin{theorem} [Soundness and Completeness of $C_n$ w.r.t. $\matGM_{C_n}$]  \label{sound-compl-Cn} \ \\
	Let $\Gamma \cup \{\varphi\} \subseteq \bF(\Sigma,\matV)$. Then: $\Gamma \vdash_{C_n} \varphi$ \ iff \ $\Gamma \vDash_{\matGM_{C_n}}^\mathsf{RN} \varphi$. 
\end{theorem}

\begin{proof}
First, suppose $\Gamma\vdash_{C_{n}}\varphi$, and take a valuation $\nu\in\mathcal{F}_{C_{n}}$ for which $\nu[\Gamma]\subseteq D_{n}$: from Lemma \ref{lem-sound-Cn} the function $\mathsf{b}:\bF(\Sigma, \mathcal{V})\rightarrow\textbf{2}$, defined by $\mathsf{b}(\alpha)=\nu(\alpha)_{1}$, is a bivaluation which, by hypothesis, satisfies $\mathsf{b}[\Gamma]\subseteq\{1\}$. Given the soundness of $C_{n}$ with respect to bivaluations and the fact that $\Gamma\vdash_{C_{n}}\varphi$, it follows that $\mathsf{b}(\varphi)=1$, and therefore $\nu(\varphi)\in D_{n}$. This shows that $\Gamma \vDash_{\matGM_{C_n}}^\mathsf{RN} \varphi$.

Conversely, suppose that $\Gamma \vDash_{\matGM_{C_n}}^\mathsf{RN} \varphi$ and let $\mathsf{b}$ be a $C_{n}$-bivaluation such that $\mathsf{b}[\Gamma]\subseteq\{1\}$. By Lemma \ref{lem-compl-Cn}, $\nu:\textbf{F}(\Sigma, \mathcal{V})\rightarrow \mathcal{A}_{C_{n}}$ defined by $\nu(\alpha)=(\mathsf{b}(\alpha), \mathsf{b}(\neg\alpha), \mathsf{b}(\alpha^{1}), \ldots , \mathsf{b}(\alpha^{n-1}))$,
is a valuation in $\mathcal{F}_{C_{n}}$ for which $\nu[\Gamma]\subseteq D_{n}$. From $\Gamma \vDash_{\matGM_{C_n}}^\mathsf{RN} \varphi$ it follows that $\nu(\varphi)\in D_{n}$, hence $\mathsf{b}(\varphi)=1$. By completeness of $C_{n}$ with respect to bivaluations, $\Gamma\vdash_{C_{n}}\varphi$.
\end{proof}

\subsection{Row-branching truth-tables for $C_{n}$, with $n \geq 2$}

As in the case for \mbccl, \cila\ and $C_1$, we can obtain decision procedures for $C_{n}$ (for $n \geq 2$) through the RNmatrix $\mathcal{RM}_{C_{n}}$. The idea is the same: given a formula $\varphi$, consider a sequence of all subformulas $\varphi_{1}, \ldots , \varphi_{k}=\varphi$ ordered by complexity (formulas with the same complexity are arranged arbitrarily). Then, it is constructed a row-branching truth-table for $\varphi$, with one column for each $\varphi_{i}$. A complex formula $\varphi_i$ can receive more than one truth-value in a given row, hence that row splits into several new ones, one for each possible value assigned to $\varphi_i$ on that row by the multioperator associated with the main connective of $\varphi_i$. In order to attend the restrictions imposed by $\mathcal{F}_{C_{n}}$, the following rules are necessary: 
\begin{enumerate}
\item if $\varphi_{i}=\varphi_{j}\wedge\neg\varphi_{j}$ for some $1\leq j< i$, and $\varphi_{j}$ takes the value $t^{n}_{0}$ on a row, $\varphi_{i}$ may only take the value $T_{n}$  on that row;
\item if $\varphi_{i}=\varphi_{j}\wedge\neg\varphi_{j}$ for some $1\leq j< i$, and $\varphi_{j}$ takes the value $t^{n}_{k}$ on a row (for $1\leq k\leq n-1$), that row splits assigning to $\varphi_{i}$ all the values of $I_{n}$;
\item if $\varphi_{i}=\neg(\varphi_{j}\wedge\neg\varphi_{j})$ for some $1\leq j< i$, and $\varphi_{j}$ takes the value $t^{n}_{k}$ on a row (for $1\leq k\leq n-1$), $\varphi_{i}$ may only assume the value $t^{n}_{k-1}$ on that row.
\end{enumerate}

A formula $\varphi$ is declared to be valid according to these truth-tables whenever its respective column only contains elements of $D_{n}$. Let us see that this constitutes a decision procedure for $C_n$, for $n \geq 2$. For any set $\Gamma_{0}\subseteq \bF(\Sigma, \mathcal{V})$ closed by subformulas, a function $\nu_{0}:\Gamma_{0}\rightarrow B_{n}$ satisfying
\begin{enumerate}
\item if $\neg\alpha\in\Gamma_{0}$, then $\nu_{0}(\neg\alpha)\in\tilde{\neg}\nu_{0}(\alpha)$;
\item if $\alpha\#\beta\in\Gamma_{0}$, for $\#\in\{\vee, \wedge, \rightarrow\}$, then $\nu_{0}(\alpha\#\beta)\in\nu_{0}(\alpha)\tilde{\#}\nu_{0}(\beta)$;
\item if $\alpha\wedge\neg\alpha\in\Gamma_{0}$ and $\nu_{0}(\alpha)=t^{n}_{0}$, then $\nu_{0}(\alpha\wedge\neg\alpha)=T_{n}$;
\item if $\alpha\wedge\neg\alpha\in\Gamma_{0}$ and $\nu_{0}(\alpha)=t^{n}_{k}$, for $1\leq k\leq n-1$, then $\nu(\alpha\wedge\neg\alpha)\in I_{n}$;
\item if $\alpha^{1}\in\Gamma_{0}$ and $\nu_{0}(\alpha)=t^{n}_{k}$, for $1\leq k\leq n-1$, then $\nu_{0}(\alpha^{1})=t^{n}_{k-1}$,
\end{enumerate}
can be extended to a homomorphism $\nu:\textbf{F}(\Sigma, \mathcal{V})\rightarrow\mathcal{A}_{C_{n}}$ on $\mathcal{F}_{C_{n}}$. The proof of this fact is done by adapting the one for Proposition~\ref{exte-hom} (by ignoring the connective $\circ$) but now, on each  step $k$, it is defined the value $\nu(\alpha)$, for every formula $\alpha$ with complexity $k$, plus the values $\nu(\neg\alpha)$, $\nu(\alpha \land \neg\alpha)$ and $\nu(\alpha^1)$ according to the restrictions in $\mathcal{F}_{C_{n}}$. 


Given $\varphi$, it can be constructed a (necessarily finite) branching truth-table for $\varphi$ in $C_n$ as indicated above.  Any row of such table corresponds to a function $\nu_0$ as above, where $\Gamma_0$ is the set formed by $\varphi$ together with all of its subformulas. Then, there exists  $\nu$ in $\mathcal{F}_{C_n}$ extending $\nu_0$ and, moreover, any 
homomorphism $\nu$ in $\mathcal{F}_{C_n}$ can be obtained by extending such mappings $\nu_0$: by restricting $\nu$ to $\Gamma_0$, it is obtained a function $\nu_{0}$ whose possible extensions to homomorphisms include $\nu$. Outside $\Gamma_0$, $\nu$ can be defined arbitrarily, while preserving the conditions for being an element of $\mathcal{F}_{C_n}$, since all the information for evaluating $\nu(\varphi)$ is contained in $\Gamma_0$. Thus,  $\varphi$ is valid in $\mathcal{RM}_{C_n}$ iff  the branching table for $\varphi$ assigns a designated value to $\varphi$ on each  row.

\section{Tableau systems for $C_n$, with $n \geq 1$}\label{TableauxforCn}

\subsection{Tableaux for $C_1$, \cila\ and \mbccl} \label{tableauxC1}

In the previous sections we have proved that RNmatrices constitute a decision procedure for the logics $C_n$, as well as for \cila\ and \mbccl, all of them uncharacterizable by finite Nmatrices. However, checking theoremhood in $n$-valued Nmatrices, for $n \geq 3$, can be a difficult task (as observed above, the RNmatrices  decide theoremhood by means of the tables constructed from the corresponding Nmatrices, of which some rows are then deleted). In this section we introduce a sound and complete tableaux system for $C_1$ constructed from the RNmatrix $\matGM_{C_1}=(\mathcal{A}_{C_1}, D, \mathcal{F}_{C_1})$ for $C_1$ introduced in Definition~\ref{defRNC1}.
In Section~\ref{tableauxCn}  we will present tableau systems for $C_n$, for $n \geq 2$, by following the same approach.  The technique for constructing the tableau systems presented here is based on the forthcoming paper~\cite{con:far:per:21}.\footnote{Recently, \cite{Pawlowski} introduced a general method for obtaining tableaux-like proof systems from a finite-valued Nmatrix. Our approach, closely related to~\cite{con:far:per:21}, follows a different direction.}   


Along this section, we will consider labelled formulas of the form $\laL(\varphi)$, for $\laL \in \{\laT,\lat,\laF\}$ and $\varphi \in \bF(\Sigma,\matV)$. The label $\laL$ represents the truth-value $L$ and in order to make the distinction between label and its truth value clearer, we use slightly different fonts.

\begin{defi} [Labelled tableaux system for $C_1$] \label{deftablT1} Let $\bbT_1$ be a labelled tableaux system for $C_1$ defined by the following rules:\footnote{The symbol $\mid$ on the consequence of the rules denotes branching, where $\mid$ separates the different branches created by the rule.}

$$
\begin{array}{ccc}
\displaystyle \frac{\laT(\neg \varphi)}{\laF(\varphi) \mid \lat(\varphi)} &  \displaystyle \frac{\lat(\neg \varphi)}{\lat(\varphi)}  &  \displaystyle \frac{\laF(\neg \varphi)}{\laT(\varphi)} \\[2mm]
&&\\[2mm]
\displaystyle \frac{\laT(\varphi \land \psi)}{\begin{array}{c|c|c|c}\laT(\varphi) & \laT(\varphi) & \lat(\varphi) & \lat(\varphi) \\ \laT(\psi) & \lat(\psi) & \laT(\psi) & \lat(\psi)\end{array}} 
&  \displaystyle \frac{\lat(\varphi \land \psi)}{\begin{array}{c|c|c}\laT(\varphi) & \lat(\varphi) & \lat(\varphi) \\ \lat(\psi) & \laT(\psi) & \lat(\psi)\end{array}} &  \displaystyle \frac{\laF(\varphi \land \psi)}{\laF(\varphi) \mid \laF(\psi)} \\[2mm]
&&\\[2mm]
\displaystyle \frac{\laT(\varphi \lor \psi)}{\begin{array}{c|c|c|c}\laT(\varphi) & \lat(\varphi) & \laT(\psi) & \lat(\psi) \end{array}} 
&  \displaystyle \frac{\lat(\varphi \lor \psi)}{\begin{array}{c|c} \lat(\varphi) & \lat(\psi)\end{array}} &  \displaystyle \frac{\laF(\varphi \lor \psi)}{\begin{array}{c}\laF(\varphi)\\ \laF(\psi) \end{array}} \\[2mm]
&&\\[2mm]
\displaystyle \frac{\laT(\varphi \to \psi)}{\begin{array}{c|c|c}\laF(\varphi) & \laT(\psi) & \lat(\psi)\end{array}} 
&  \displaystyle \frac{\lat(\varphi \to \psi)}{\begin{array}{c|c}\lat(\varphi) & \lat(\psi) \\ \laT(\psi) & \end{array}} &  \displaystyle \frac{\laF(\varphi \to \psi)}{\begin{array}{c|c}\laT(\varphi) & \lat(\varphi) \\ \laF(\psi) & \laF(\psi)\end{array}} \\[2mm]
\end{array}
$$
A branch $\theta$ of a tableau in $\bbT_1$ for a signed formula $\laL(\varphi)$ is said to be {\em closed} if it contains two signed formulas $\laL(\varphi)$ and $\laL'(\varphi)$ such that $\laL\neq\laL'$, or if it contains a signed formula $\lat(\psi \land \neg \psi)$.  A branch $\theta$ is {\em complete} if, for every signed formula $\laL(\psi)$ occurring in $\theta$, $\theta$ contains all the formulas of one of the branches resulting from the application of the tableau rule for $\laL(\psi)$.\footnote{Observe that, for every non-atomic signed formula $\laL(\psi)$, there exists one and only one rule in $\bbT_1$ applicable to $\laL(\psi)$.} A complete branch is {\em open} if it is not closed.
A tableau in $\bbT_1$ is {\em closed} if it contains a closed branch. 
A tableau  in $\bbT_1$ is {\em completed} if every branch is either closed or complete. A completed tableau is {\em open} if it is not closed.
\end{defi}

\noindent The rules of $\bbT_1$ are analytic, in the sense that, if $\laL'(\psi)$  is a labelled formula appearing in the consequence of a rule for $\laL(\varphi)$, then $\psi$ is a strict subformula of $\varphi$. Because of this,  any tableau starting with $\laL(\varphi)$ will  be completed in a finite number of steps, thus producing a decision procedure for $C_1$, as it will be shown below. The proof of soundness and completeness of $\bbT_1$ w.r.t. the RNmatrix $\matGM_{C_1}$ (hence, w.r.t. $C_1$) to be presented here closely follows the lines of  the  book~\cite{Smullyan}, a standard reference for tableaux systems for classical logic.

\begin{defi} \label{derivT1} A formula $\varphi$ over $\Sigma$ is said to be {\em provable} by tableaux in $\bbT_1$, denoted by  $\vdash_{\bbT_1} \varphi$, if there exists a closed tableau in $\bbT_1$ starting from $\laF(\varphi)$. Given a finite set $\Gamma=\{\gamma_1,\ldots,\gamma_n\} \subseteq \bF(\Sigma,\matV)$, $\varphi$ is said to be {\em provable from $\Gamma$} by tableaux in  $\bbT_1$, denoted by  $\Gamma \vdash_{\bbT_1} \varphi$, if $(\gamma_1 \to (\gamma_2 \to \ldots \to(\gamma_n \to \varphi) \ldots ))$ is provable by tableaux in  $\bbT_1$.
\end{defi}

\begin{defi} \label{valsat} Let $\nu$ be  a valuation in $\mathcal{F}_{C_1}$. We say that a signed formula $\laL(\varphi)$ is {\em true} in $\nu$, or $\nu$ {\em satisfies} $\laL(\varphi)$, if $\nu(\varphi)=L$; otherwise, it is {\em false}  in $\nu$. A branch $\theta$ of a tableau $\calT$ is {\em true under $\nu$}, or $\nu$ {\em satisfies} $\theta$, if every signed formula occurring in $\theta$ is true in $\nu$. A tableau $\calT$ is {\em true under $\nu$}, or $\nu$ {\em satisfies} $\calT$,  if some branch of it is true under $\nu$.
\end{defi}

\begin{remark} \label{closed-unsat} Observe that, if $\nu \in \mathcal{F}_{C_1}$ satisfies $\laL(\varphi)$, then $\laL(\varphi) \neq \lat(\psi \land \neg\psi)$. On the other hand, by definition $\nu$ cannot satisfy a branch $\theta$ containing $\laL(\varphi)$ and $\laL'(\varphi)$ for $\laL\neq\laL'$. From this, a closed branch is unsatisfiable,  and so is a closed tableau.
\end{remark} 

\begin{lemma} \label{satrulTC1} Let $\laL(\varphi)$ be a signed formula  where $\varphi$ is non-atomic, and let $R$ be the unique rule in $\bbT_1$ applicable to it. If $\nu$ is a valuation in $\mathcal{F}_{C_1}$ satisfying  $\laL(\varphi)$, then $\nu$ satisfies all the formulas of at least one of the branches resulting from the application of the tableau rule $R$ for $\laL(\varphi)$.
\end{lemma}
\begin{proof}
It is immediate from the definition of the tableau rules  in $\bbT_1$, the definition of the multioperators in $\mathcal{A}_{C_1}$, and by  Definition~\ref{valsat}.
\end{proof}

\begin{theorem} [Soundness of $\bbT_1$ w.r.t. $C_1$] \label{soundTC1}
Let $\Gamma \cup \{\varphi\} \subseteq \bF(\Sigma,\matV)$ be a finite set of formulas. Then, 
if $\Gamma \vdash_{\bbT_1} \varphi$, it follows that $\Gamma \vDash_{\matGM_{C_1}}^\mathsf{RN} \varphi$. 
\end{theorem}
\begin{proof}
By Definition~\ref{derivT1}, and since $\vDash_{\matGM_{C_1}}^\mathsf{RN}$  satisfies the deduction metatheorem (by  soundness and completeness w.r.t. $C_1$), it suffices to prove the result for $\Gamma=\emptyset$. A completed tableau $\calT$ for $\laF(\varphi)$ is obtained by means of a finite sequence of tableaux $\calT_0, \, \ldots, \, \calT_k=\calT$. Thus, $\calT_0$ only contains $\laF(\varphi)$ and, for every $0 \leq n \leq k-1$, $\calT_{n+1}$ is obtained from $\calT_n$ by applying some rule of $\bbT_1$ to a signed formula $\laL(\psi)$ that has not yet been used,  occurring in a branch $\theta$ of $\calT_n$.  Now, fix  $\nu \in \mathcal{F}_{C_1}$.\\
{\bf Fact:}  If $\nu$ satisfies $\calT_n$ then it also  satisfies $\calT_{n+1}$, for every $0 \leq n \leq k-1$.\\ 
Indeed, if $\nu \in \mathcal{F}_{C_1}$ satisfies $\calT_n$ then it satisfies some branch $\theta^\prime$ of $\calT_n$. If $\theta=\theta^\prime$ then, in particular, $\nu$ satisfies $\laL(\psi)$. Thus, $\nu$ satisfies all the formulas of at least one of the branches resulting from the application of the tableau rule $R$ for $\laL(\varphi)$, by Lemma~\ref{satrulTC1}, which implies that $\nu$ satisfies at least one of the branches resulting from the expansion of $\theta$ by applying $R$ to  $\laL(\psi)$. That is, $\nu$ satisfies $\calT_{n+1}$. On the other hand, if  $\theta\neq \theta^\prime$ then $\theta^\prime$ is still a branch of  $\calT_{n+1}$, and so  $\calT_{n+1}$ contains a branch satisfied by $\nu$. Then, $\nu$ satisfies $\calT_{n+1}$ also in this case. This proves the {\bf Fact}.


Finally, suppose that $\nvDash_{\matGM_{C_1}}^\mathsf{RN} \varphi$. This means that there exists some  $\nu \in \mathcal{F}_{C_1}$ such that $\nu(\varphi)=F$. Hence, $\nu$ satisfies   $\laF(\varphi)$ and so, by the {\bf Fact}, $\nu$ satisfies any completed tableau $\calT$ for $\laF(\varphi)$. By Remark~\ref{closed-unsat}, $\calT$ cannot be  closed. That is, every completed tableau for $\laF(\varphi)$ is open, and so  $\not\vdash_{\bbT_1} \varphi$.
\end{proof}

\

\noindent The proof of completeness require the use of Hintikka sets.

\begin{defi} \label{hintikkaTC1}
A nonempty set $\Gamma$ of signed formulas over $\Sigma$ is a {\em Hintikka set} for $\bbT_1$ if it satisfies the following conditions:
\begin{enumerate}
\item If $\laL(\varphi)$ and $\laL'(\varphi)$ belong to $\Gamma$, then $\laL=\laL'$.
\item For every $\varphi \in \bF(\Sigma,\matV)$, $\lat(\varphi \land \neg \varphi)$ does not belong to $\Gamma$.
\item If $\laT(\neg \varphi)$ belongs to $\Gamma$ then either $\laF(\varphi)$ belongs to $\Gamma$ or $\lat(\varphi)$ belongs to $\Gamma$.
\item If $\lat(\neg \varphi)$ belongs to $\Gamma$, then $\lat(\varphi)$ belongs to $\Gamma$.
\item If $\laF(\neg \varphi)$ belongs to $\Gamma$, then  $\laT(\varphi)$ belongs to $\Gamma$.
\item If $\laT(\varphi \land \psi)$ belongs to $\Gamma$ then: either  $\laT(\varphi)$ and $\laT(\psi)$ belong to $\Gamma$, or $\laT(\varphi)$ and $\lat(\psi)$ belong to $\Gamma$, or $\lat(\varphi)$ and $\laT(\psi)$  belong to $\Gamma$, or  $\lat(\varphi)$ and $\lat(\psi)$  belong to $\Gamma$.
\item If $\lat(\varphi \land \psi)$ belongs to $\Gamma$ then: either $\laT(\varphi)$ and $\lat(\psi)$ belong to $\Gamma$, or $\lat(\varphi)$ and $\laT(\psi)$  belong to $\Gamma$, or  $\lat(\varphi)$ and $\lat(\psi)$  belong to $\Gamma$.
\item If $\laF(\varphi \land \psi)$ belongs to $\Gamma$ then: either $\laF(\varphi)$  belongs to $\Gamma$ or $\laF(\psi)$ belongs to $\Gamma$.
\item If $\laT(\varphi \lor \psi)$ belongs to $\Gamma$ then: either  $\laT(\varphi)$ belongs to $\Gamma$, or $\lat(\varphi)$ belongs to $\Gamma$, or $\laT(\psi)$  belongs to $\Gamma$, or  $\lat(\psi)$  belongs to $\Gamma$.
\item If $\lat(\varphi \lor \psi)$ belongs to $\Gamma$ then: either $\lat(\varphi)$ belongs to $\Gamma$, or  $\lat(\psi)$  belongs to $\Gamma$.
\item If $\laF(\varphi \lor \psi)$ belongs to $\Gamma$ then $\laF(\varphi)$ and  $\laF(\psi)$  belong to $\Gamma$.
\item If $\laT(\varphi \to \psi)$ belongs to $\Gamma$ then: either  $\laF(\varphi)$   belongs to $\Gamma$, or $\laT(\psi)$ belongs to $\Gamma$, or $\lat(\psi)$  belongs to $\Gamma$.
\item If $\lat(\varphi \to \psi)$ belongs to $\Gamma$ then: either $\lat(\varphi)$ and $\laT(\psi)$  belong to $\Gamma$, or  $\lat(\psi)$  belongs to $\Gamma$.
\item If $\laF(\varphi \to \psi)$ belongs to $\Gamma$ then: either $\laT(\varphi)$ and $\laF(\psi)$ belong to $\Gamma$, or $\lat(\varphi)$ and $\laF(\psi)$ belong to $\Gamma$.
\end{enumerate}
\end{defi}

\

\noindent The next step is to show that any Hintikka set is satisfiable in $\mathcal{F}_{C_1}$. In order to do this, let us fix a Hintikka set $\Gamma$ for $\bbT_1$. Let $\Gamma_0 = \{ \varphi \in {\bf F}(\Sigma,\mathcal{V}) \ : \ \laL(\varphi) \in \Gamma\}$, and let $\nu_0:\Gamma_0\rightarrow\{T,t,F\}$ be a function such that $\nu_0(\varphi)=L$ iff  $\laL(\varphi) \in \Gamma$. 
Observe that, by item~1 of Definition~\ref{hintikkaTC1}, $\nu_0$ is well-defined. In principle, it should be possible to define a homomorphism $\nu$ in $\mathcal{F}_{C_1}$ extending $\nu_0$, as it was done in the proof of Proposition~\ref{exte-hom}. However, there is a big difference here: the set $\Gamma_0$ is not necessarily closed under subformulas. Then, it is possible to have $\alpha \notin \Gamma_0$ and $\beta \in \Gamma_0$ such that $\alpha$ is a subformula of $\beta$. In such cases, the value $\nu(\alpha)$ to be assigned  to $\alpha$, which is defined with some degree of arbitrariness, could (in principle) be incompatible with the already given value $\nu_0(\beta)$, provided that $\nu(\beta)$ must coincide with $\nu_0(\beta)$. It will be argued that, since $\Gamma$ is a Hintikka set, no conflict will occur.

Indeed, observe that the only cases in which $\beta \in \Gamma_0$ but $\alpha \not\in\Gamma_0$, for some subformula $\alpha$ of $\beta$, are originated by clauses 8, 9, 10, 12 and 13 from Definition~\ref{hintikkaTC1}. Let us analyze, for instance, clause 9. Thus, suppose that $\laT(\alpha \vee \beta) \in \Gamma$, $\laT(\alpha) \in \Gamma$ but $\laL(\beta) \notin \Gamma$, for every $\laL$. Then, $\alpha \vee \beta \in \Gamma_0$, $\alpha \in \Gamma_0$ but $\beta\notin\Gamma_0$. The values $\nu(\alpha)$ and $\nu(\alpha \vee \beta)$ are automatically given by $\nu_0$. However, the value $\nu(\beta)$ could be arbitrarily defined (under certain restrictions), according  to the procedure given in the proof of Proposition~\ref{exte-hom}. One wonders if some bad choice for the value $\nu(\beta)$, say $a$, could produce the undesired situation $\nu(\alpha \vee \beta) \notin \nu(\alpha) \,\tilde{\vee} \, a = \nu(\alpha) \,\tilde{\vee} \, \nu(\beta)$. Fortunately, this will not be the case: one of the facts that clause 9 reflects  is that, in $\mathcal{A}_{C_1}$, $T \in T \, \tilde{\vee} \, a$ for every $a \in \{T,t,f\}$.  This shows that any choice $a$ for $\nu(\beta)$ will produce  that $\nu(\alpha \vee \beta) \in \nu(\alpha) \,\tilde{\vee} \, a = \nu(\alpha) \,\tilde{\vee} \, \nu(\beta)$, since $\nu(\alpha \vee \beta)= \nu(\alpha)=T$ in this case. Another fact reflected by clause 9 is that, in $\mathcal{A}_{C_1}$, $T \in a \, \tilde{\vee} \, t$ for every $a \in \{T,t,f\}$. Thus, if $\laT(\alpha \vee \beta) \in \Gamma$ and $\lat(\beta) \in \Gamma$, but $\laL(\alpha) \notin \Gamma$  for every $\laL$, then any value $a$ chosen for $\nu(\alpha)$ will produce that $\nu(\alpha \vee \beta) \in a \,\tilde{\vee} \, \nu(\beta) = \nu(\alpha) \,\tilde{\vee} \, \nu(\beta)$, as required. By a similar reasoning, it can be seen that, in spite of $\Gamma_0$  not being closed under subformulas (because of the clauses  from Definition~\ref{hintikkaTC1} above mentioned), it is possible to define $\nu$ in a similar way as it was done in the proof of Proposition~\ref{exte-hom}, without any conflicts. This lead us to the following:

\begin{prop} \label{extenv}
Let $\Gamma$ be a Hintikka set for $\bbT_1$. Let $\Gamma_0 = \{ \varphi \in {\bf F}(\Sigma,\mathcal{V}) \ : \ \laL(\varphi) \in \Gamma\}$, and let $\nu_0:\Gamma_0\rightarrow\{T,t,F\}$ be a function defined as follows: $\nu_0(\varphi)=L$ iff  $\laL(\varphi) \in \Gamma$.
Then, there exists a homomorphism $\nu$ in $\mathcal{F}_{C_1}$ extending $\nu_0$, i.e., such that $\nu(\alpha)=\nu_0(\alpha)$ for every $\alpha \in \Gamma_0$.
\end{prop} 
\begin{proof}
As observed above, $\nu_0$ is well-defined. By the considerations above, the proof can be obtained by adapting the one given for Proposition~\ref{exte-hom}, taking into account that $\cons$ does not belong to the signature $\Sigma$. Hence, the valuation $\nu$ can be defined by induction on the complexity of formulas as follows: on each  step $n \geq 0$, it is defined the value $\nu(\alpha)$ for every formula $\alpha$ with complexity $n$, together with the values $\nu(\neg\alpha)$ and $\nu(\alpha \land \neg\alpha)$. The method for defining the function $\nu$ is the one described in the proof of Proposition~\ref{exte-hom}. 
Then, it follows that $\nu:{\bf F}(\Sigma,\mathcal{V})\rightarrow\{T,t,F\}$  is a function which satisfies the desired properties.
\end{proof}

\begin{theorem} [Hintikka's Lemma for $\bbT_1$] \label{hintikka-lemmaTC1} Let $\Gamma$ be a Hintikka set for $\bbT_1$. Then, there exists a homomorphism $\nu$ in $\mathcal{F}_{C_1}$ such that $\laL(\varphi)$ is true in $\nu$ for every $\laL(\varphi) \in \Gamma$.
\end{theorem}
\begin{proof} Let $\Gamma_0 = \{ \varphi \in {\bf F}(\Sigma,\mathcal{V}) \ : \ \laL(\varphi) \in \Gamma\}$, and let $\nu_0:\Gamma_0\rightarrow\{T,t,F\}$ be a function such that $\nu_0(\varphi)=L$ iff  $\laL(\varphi) \in \Gamma$. By Proposition~\ref{extenv}, there exists a homomorphism $\nu$ in $\mathcal{F}_{C_1}$  extending $\nu_0$. This means that $\laL(\varphi)$ is true in $\nu$ for every $\laL(\varphi) \in \Gamma$.
\end{proof}

\begin{prop} \label{open-hintikka} Let $\laL_0(\varphi_0)$  be a signed formula over $\Sigma$.
Let $\theta$ be an open branch of a completed tableau $\calT$ in  $\bbT_1$ for $\laL_0(\varphi_0)$, and let $\Gamma$ be the set of signed formulas occurring in  $\theta$.
Then, $\Gamma$ is a Hintikka set for $\bbT_1$.
\end{prop}
\begin{proof}
Since $\theta$ is open then, by Definition~\ref{deftablT1}, if $\laL(\varphi)$ and $\laL'(\varphi)$ belong to $\Gamma$ then $\laL=\laL'$. In addition, $t(\varphi \land \neg \varphi) \notin\Gamma$. This shows that $\Gamma$ satisfies clauses~1 and~2 of Definition~\ref{hintikkaTC1}. If $\laL(\varphi) \in \Gamma$ for $\varphi$ of the form $\neg\psi$ or $\gamma \,\#\, \psi$, for some $\# \in \{\land,\lor,\to\}$, then, by the tableau rules for $\bbT_1$, and taking into consideration that $\calT$ is a completed tableau, necessarily $\laL(\varphi)$ was used at some stage of the procedure for defining $\theta$. Hence  it is immediate to see that clauses~3-~14  of Definition~\ref{hintikkaTC1} are satisfied.  From this, $\Gamma$ is a Hintikka set for  $\bbT_1$. 
\end{proof}

\begin{coro} \label{sat-branchTC1} Let $\laL(\varphi)$  be a signed formula over $\Sigma$. Let $\theta$ be an open branch of a completed tableau $\calT$ in  $\bbT_1$ for $\laL(\varphi)$, and let $\Gamma$ be the set of signed formulas occurring in  $\theta$. Then, there exists a homomorphism $\nu$ in $\mathcal{F}_{C_1}$ such that $\laL(\varphi)$ is true in $\nu$ for every $\laL(\varphi) \in \Gamma$.
\end{coro}
\begin{proof}
It is a consequence of Proposition~\ref{open-hintikka} and Theorem~\ref{hintikka-lemmaTC1}. 
\end{proof}

\begin{theorem} [Completeness of $\bbT_1$ w.r.t. $C_1$] \label{compleTC1}
Let $\Gamma \cup \{\varphi\} \subseteq \bF(\Sigma,\matV)$ be a finite set of formulas. Then: if $\Gamma \vDash_{\matGM_{C_1}}^\mathsf{RN} \varphi$, it follows that $\Gamma \vdash_{\bbT_1} \varphi$. 
\end{theorem}
\begin{proof} As argued at the beginning of the proof of Theorem~\ref{soundTC1}, it suffices to prove the result for $\Gamma=\emptyset$. Thus, let $\calT$ be a completed tableau in $\bbT_1$ for $\laF(\varphi)$. If $\calT$ has an open branch $\theta$ then the set $\Gamma$ of signed formulas occurring in  $\theta$  is simultaneously satisfiable by a homomorphism $\nu$ in $\mathcal{F}_{C_1}$, by Corollary~\ref{sat-branchTC1}. In particular, $\laF(\varphi)$ is true in $\nu$, which implies that $\nu(\varphi)=F$. This means that $\nvDash_{\matGM_{C_1}}^\mathsf{RN} \varphi$. From this, if $\vDash_{\matGM_{C_1}}^\mathsf{RN} \varphi$, then every completed tableau for $\laF(\varphi)$ in $\bbT_1$ is  closed. Given that it is always possible to construct a completed tableaux for $\laF(\varphi)$ in $\bbT_1$ in a finite number of steps, we obtain the following: if $\vDash_{\matGM_{C_1}}^\mathsf{RN} \varphi$ then there exists a completed closed tableau for $\laF(\varphi)$. That is, $\varphi$ is provable by tableaux in $\bbT_1$.
\end{proof}

\begin{coro} [$\bbT_1$ as a decision procedure for $C_1$ based on $\matGM_{C_1}$] \ \\
If $\varphi$ is valid in $C_1$ then every completed tableau for $\laF(\varphi)$ is closed. If $\varphi$ is not valid in $C_1$, then every completed tableau for $\laF(\varphi)$ is open. In this case, any open branch of any completed tableau for $\laF(\varphi)$ gives us a valuation $\nu$ in $\mathcal{F}_{C_1}$ such that $\nu(\varphi)=F$.
\end{coro}

\noindent
With minor modifications, it is easy to obtain from $\bbT_1$  sound and complete tableau systems $\bbT_\cila$ and $\bbT_\mbccl$ for \cila\ and \mbccl, respectively. In the case of $\bbT_\cila$, it is enough adding to $\bbT_1$ the following rules:

$$
\begin{array}{cc}
\displaystyle \frac{\laT(\cons \varphi)}{\laT(\varphi) \mid \laF(\varphi)} & \hspace{6mm}  \displaystyle \frac{\laF(\cons \varphi)}{\lat(\varphi)} \\[2mm]
&\\[2mm]
\end{array}
$$

\noindent as well as modifying the definition of closed branch as follows: a branch $\theta$ for a signed formula $\laL(\varphi)$ is said to be {\em closed} in $\bbT_\cila$ if it contains two signed formulas $\laL(\psi)$ and $\laL'(\psi)$ such that $\laL\neq\laL'$, or if it contains a signed formula $\lat(\psi \land \neg \psi)$, or if it contains a signed formula $\lat(\cons\psi)$. In the case of $\bbT_\mbccl$, it is enough modifying the rules of $\bbT_\cila$ according to the multioperators of $\mathcal{A}_{\mbccl}$ (recalling that $\mathcal{A}_{\cila}$ is a submultialgebra of $\mathcal{A}_{\mbccl}$). For instance, the rules for $\cons$ are  defined as follows:

$$
\begin{array}{ccc}
\displaystyle \frac{\laT(\cons \varphi)}{\laT(\varphi) \mid \laF(\varphi)} & \hspace{6mm} \displaystyle \frac{\lat(\cons \varphi)}{\laT(\varphi) \mid \laF(\varphi)} & \hspace{6mm}  \displaystyle \frac{\laF(\cons \varphi)}{\lat(\varphi)} \\[2mm]
&\\[2mm]
\end{array}
$$

\noindent In turn, the definition of closed branch in $\bbT_\mbccl$ is as in  $\bbT_1$.

\subsection{Tableaux for $C_n$, with $n \geq 2$} \label{tableauxCn}

As one would perhaps expect, the tableau systems induced by the RNmatrices for the logics $C_{n}$, for $n\geq 2$, are quite similar to the ones presented for the system $C_{1}$. Here, we will give a brief description of them, without entering in technicals details, given the similarity to $C_{1}$'s case.

Now, a wider universe of labelled formulas $\textsf{L}(\varphi)$ will be considered, with labels $\textsf{L}$ in $\mathbb{B}_{n}=\{\laT_{n}, \lat^{n}_{0}, \ldots , \lat^{n}_{n-1}, \laF_{n}\}$ and formulas in $\textbf{F}(\Sigma, \mathcal{V})$. For simplicity, we will need a slightly more general notation,  in order to deal with  larger tableaux rules. Let $\textsf{X}$ and $\textsf{Y}$ be sets of labels, and consider the rules below. 
$$
\begin{array}{ccc}
\displaystyle \frac{\textsf{L}(\varphi)}{\textsf{L}^\prime(\psi)\mid \textsf{X}(\psi)} &  \displaystyle \frac{\textsf{L}(\varphi)}{\begin{array}{c}\textsf{L}^{\prime}(\gamma) \\ \textsf{X}(\psi)\end{array}}  &  \displaystyle \frac{\textsf{L}(\varphi)}{\begin{array}{c|c|c} \textsf{L}^{\prime}(\gamma) & \textsf{L}^{\prime\prime}(\psi) & \textsf{X}(\gamma)\\ \textsf{Y}(\psi) & \textsf{X}(\gamma) & \textsf{Y}(\psi)\end{array}} \\[2mm]
\end{array}
$$

Suppose that $\textsf{X}$ and $\textsf{Y}$ have $p$ and $q$ elements, respectively. The leftmost rule states that the branch splits into $p+1$ branches: $\textsf{L}^{\prime}(\psi)$ and $\textsf{x}(\psi)$, for every $\textsf{x}\in \textsf{X}$. The rule on the center states that the branch splits into $p$ branches, each of them contaning $\textsf{L}^{\prime}(\gamma)$ and $\textsf{x}(\psi)$, for some $\textsf{x}\in\textsf{X}$. Finally, the rightmost rule splits the branch into $q+p+pq$ branches: each of the leftmost $q$ ones contains $\textsf{L}^{\prime}(\gamma)$ and $\textsf{y}(\psi)$ for some $\textsf{y}\in\textsf{Y}$; each of the following $p$ ones contains $\textsf{L}^{\prime\prime}(\psi)$  and $\textsf{x}(\gamma)$ for some $\textsf{x}\in\textsf{X}$; and each of the final $pq$ ones contains  $\textsf{x}(\gamma)$ and $\textsf{y}(\psi)$ for some $\textsf{x}\in \textsf{X}$ and some $\textsf{y}\in \textsf{Y}$.

\begin{defi} [Labelled tableaux system for $C_n$, for $n\geq 2$] \label{deftablTn} Let $\bbT_n$ be a labelled tableaux system for $C_n$ defined by the following rules, where $i$ takes value among $\{0, 1, \ldots , n-1\}$, $\textsf{I}_{n}$ is the set of labels $\{\lat^{n}_{0}, \ldots , \lat^{n}_{n-1}\}$ and $\textsf{D}_{n}=\textsf{I}_{n}\cup\{\textsf{T}_{n}\}$:
$$\begin{array}{ccc}
\displaystyle \frac{\laT_{n}(\neg \varphi)}{\textsf{I}_{n}(\varphi) \mid \laF_{n}(\varphi)} &  \displaystyle \frac{\lat^{n}_{i}(\neg \varphi)}{\textsf{I}_{n}(\varphi)}  &  \displaystyle \frac{\laF_{n}(\neg \varphi)}{\laT_{n}(\varphi)} \\[2mm]
&&\\[2mm]
\displaystyle \frac{\laT_{n}(\varphi \land \psi)}{\begin{array}{c}\textsf{D}_{n}(\varphi) \\ \textsf{D}_{n}(\psi)\end{array}} 
&  \displaystyle \frac{\lat^{n}_{i}(\varphi \land \psi)}{\begin{array}{c|c|c}\laT_{n}(\varphi) & \textsf{I}_{n}(\varphi) & \laT_{n}(\psi) \\ \textsf{I}_{n}(\psi) & \textsf{I}_{n}(\psi) & \textsf{I}_{n}(\varphi)\end{array}} &  \displaystyle \frac{\laF_{n}(\varphi \land \psi)}{\laF_{n}(\varphi) \mid \laF_{n}(\psi)} \\[2mm]
&&\\[2mm]
\displaystyle \frac{\laT_{n}(\varphi \lor \psi)}{\textsf{D}_{n}(\varphi)\mid\textsf{D}_{n}(\psi)} 
&  \displaystyle \frac{\lat^{n}_{i}(\varphi \lor \psi)}{\textsf{I}_{n}(\varphi)\mid\textsf{I}_{n}(\psi)} &  \displaystyle \frac{\laF_{n}(\varphi \lor \psi)}{\begin{array}{c}\laF_{n}(\varphi)\\ \laF_{n}(\psi) \end{array}} \\[2mm]
&&\\[2mm]
\displaystyle \frac{\laT_{n}(\varphi \to \psi)}{\textsf{F}_{n}(\varphi)\mid\textsf{D}_{n}(\psi)} 
&  \displaystyle \frac{\lat^{n}_{i}(\varphi \to \psi)}{\begin{array}{c|c}\laT_{n}(\psi) & \textsf{I}_{n}(\psi) \\ \textsf{I}_{n}(\varphi) & \end{array}} &  \displaystyle \frac{\laF_{n}(\varphi \to \psi)}{\begin{array}{c}\laF_{n}(\psi) \\ \textsf{D}_{n}(\varphi) \end{array}} \\[2mm]
&&\\[2mm]
\end{array}
$$
A branch $\theta$ of a tableau in $\bbT_n$ is said to be {\em closed} if:
\begin{enumerate}
\item it contains two signed formulas $\laL(\varphi)$ and $\laL'(\varphi)$ such that $\laL\neq\laL'$;
\item it contains $\lat^{n}_{0}(\psi)$ and $\lat^{n}_{k}(\psi \land \neg \psi)$, for some formula $\psi$ and $0\leq k\leq n-1$;
\item it contains $\lat^{n}_{k}(\psi)$ and $\laT_{n}(\psi\wedge\neg\psi)$, or $\lat^{n}_{k}(\psi)$ and $\textsf{L}(\psi^{1})$ with $\textsf{L}\neq \lat^{n}_{k-1}$, for some formula $\psi$ and $1\leq k\leq n-1$.
\end{enumerate}
Complete and open branches, as well as closed, complete and open tableaux, are defined in $\bbT_{n}$ as they were defined in $\bbT_{1}$.
\end{defi}

A formula $\varphi$ over $\Sigma$ is said to be {\em provable} by tableaux in $\bbT_{n}$, denoted by $\vdash_{\bbT_{n}}\varphi$, if there exists a closed tableau in $\bbT_{n}$ starting from $\laF_{n}(\varphi)$. The definition of derivations in  $\bbT_{n}$ from premises is as in $\bbT_{1}$.

Next step is proving soundness and completeness of $\bbT_{n}$ w.r.t. the RNmatrix semantics defined by $\matGM_{C_n}$, for $n \geq 2$. Clearly, it suffices proving that $\vdash_{\bbT_n} \varphi$  \ iff \ $\vDash_{\matGM_{C_n}}^\mathsf{RN} \varphi$, for every $\varphi$. The notions from Definition~\ref{valsat} can be easily adapted to  the case $n \geq 2$. Note that, for any signed formula $\textsf{L}(\varphi)$ where $\varphi$ is non-atomic, there is exactly one rule $R$ in $\bbT_{n}$ applicable to it.

The proof of soundness follows the steps of the one given for $\bbT_{1}$. Thus, if $\nu \in \mathcal{F}_{C_{n}}$ satisfies $\textsf{L}(\varphi)$ and $R$ is (the only rule) appliable to it, then $\nu$ satisfies all the formulas of at least one of the branches generated by the application of $R$ to $\textsf{L}(\varphi)$. Let $\mathcal{T}_{0}, \ldots , \mathcal{T}_{m}=\mathcal{T}$ be the sequence of tableaux starting from $\laF_{n}(\varphi)$ ending with a complete tableau, as in Theorem~\ref{soundTC1} for $\bbT_{1}$. Thus, if $\nu\in\mathcal{F}_{C_{n}}$  satisfies $\mathcal{T}_{k}$ then it also satisfies $\mathcal{T}_{k+1}$. From this,  if  $\nvDash_{\matGM_{C_n}}^\mathsf{RN} \varphi$  then there exists $\nu$ satisfying $\laF_{n}(\varphi)$, hence $\nu$ satisfies any completed tableau  $\mathcal{T}$ for $\laF_{n}(\varphi)$, and so  $\mathcal{T}$ must be open. This means that  $\nvdash_{\bbT_n} \varphi$.

Now,  completeness of $\bbT_{n}$ will be stated. A non-empty set $\Gamma$ of labelled formulas (with labels in $\mathbb{B}_{n}$) is a {\em Hintikka set} for $\bbT_{n}$ if it satisfies the following.
\begin{enumerate}
\item If $\textsf{L}(\varphi)$ and $\textsf{L}^{\prime}(\varphi)$ are in $\Gamma$, $\textsf{L}=\textsf{L}^{\prime}$.
\item If $\lat^{n}_{0}(\varphi)$ is in $\Gamma$, none of $\lat^{n}_{0}(\varphi\wedge\neg\varphi), \ldots , \lat^{n}_{n-1}(\varphi\wedge\neg\varphi)$ is in $\Gamma$.
\item If $\lat^{n}_{k+1}(\varphi)$ is in $\Gamma$: $\laT_{n}(\varphi\wedge\neg\varphi)$ is not in $\Gamma$ and, if $\textsf{L}(\varphi^{1})$ is in $\Gamma$, $\textsf{L}=\lat^{n}_{k}$.
\item If $\laT_{n}(\neg\varphi)$ is in $\Gamma$, at least one of $\lat^{n}_{0}(\varphi), \ldots , \lat^{n}_{n-1}(\varphi), \laF_{n}(\varphi)$ is in $\Gamma$.
\item If $\lat^{n}_{i}(\neg\varphi)$ is in $\Gamma$, then $\lat^{n}_{j}(\varphi)$ is also in $\Gamma$,  for some $0\leq j\leq n-1$.
\item If $\laF_{n}(\neg\varphi)$ is in $\Gamma$, so is $\laT_{n}(\varphi)$.
\item If $\laT_{n}(\varphi\wedge\psi)\in \Gamma$: either $\laT_{n}(\varphi)\in\Gamma$ and $\laT_{n}(\psi)\in\Gamma$; or $\laT_{n}(\varphi)\in\Gamma$ and $\lat^{n}_{j}(\psi)\in \Gamma$ for some $0\leq j\leq n-1$; or $\laT_{n}(\psi)\in\Gamma$ and $\lat^{n}_{i}(\varphi) \in\Gamma$, for some $0\leq i\leq n-1$; or $\lat^{n}_{i}(\varphi)\in \Gamma$ and $\lat^{n}_{j}(\psi)\in \Gamma$, for some $0\leq i,j\leq n-1$.
\item If $\lat^{n}_{i}(\varphi\wedge\psi)\in\Gamma$: either $\laT_{n}(\varphi)\in\Gamma$ and $\lat^{n}_{j}(\psi)\in\Gamma$, for some $0\leq j\leq n-1$; or $\laT_{n}(\psi)\in\Gamma$ and  $\lat^{n}_{k}(\varphi) \in\Gamma$, for some $0\leq k\leq n-1$; or $\lat^{n}_{k}(\varphi)\in \Gamma$ and $\lat^{n}_{j}(\psi)\in \Gamma$, for some $0\leq k,j\leq n-1$.
\item If $\laF_{n}(\varphi\wedge\psi)$ is in $\Gamma$, either $\laF_{n}(\varphi)\in \Gamma$ or $\laF_{n}(\psi) \in \Gamma$.
\item If $\laT_{n}(\varphi\vee\psi)\in\Gamma$: either $\laT_{n}(\varphi) \in\Gamma$; or $\laT_{n}(\psi) \in\Gamma$; or $\lat^{n}_{i}(\varphi)\in\Gamma$, for some $0\leq i\leq n-1$; or $\lat^{n}_{j}(\psi)\in\Gamma$, for some $0\leq j\leq n-1$.
\item If $\lat^{n}_{i}(\varphi\vee\psi)\in\Gamma$: either $\lat^{n}_{k}(\varphi)\in\Gamma$, for some $0\leq k\leq n-1$; or $\lat^{n}_{j}(\psi)\in\Gamma$, for some $0\leq j\leq n-1$.
\item If $\laF_{n}(\varphi\vee\psi)$ is in $\Gamma$, then both $\laF_{n}(\varphi)$ and $\laF_{n}(\psi)$ are in $\Gamma$.
\item If $\laT_{n}(\varphi\rightarrow\psi)$ is in $\Gamma$: either $\laF_{n}(\varphi)\in\Gamma$; or $\laT_{n}(\psi)\in\Gamma$; or $\lat^{n}_{i}(\psi)\in\Gamma$, for some $0\leq i\leq n-1$.
\item If $\lat^{n}_{i}(\varphi\rightarrow\psi)\in\Gamma$: either $\laT_{n}(\psi)\in\Gamma$ and $\lat^{n}_{k}(\varphi)\in\Gamma$, for $0\leq k\leq n-1$; or $\lat^{n}_{j}(\psi)\in\Gamma$, for $0\leq j\leq n-1$.
\item If $\laF_{n}(\varphi\rightarrow\psi)\in\Gamma$: either $\laT_{n}(\varphi)\in\Gamma$ and $\laF_{n}(\psi)\in\Gamma$; or $\lat^{n}_{i}(\varphi)\in\Gamma$ and $\laF_{n}(\psi)\in\Gamma$, for some $0\leq i\leq n-1$.
\end{enumerate}

As it was done with $\bbT_{1}$, the next step is to prove  that a Hintikka set $\Gamma$ for $\bbT_{n}$ is satisfiable in $\mathcal{F}_{C_{n}}$. That is, there exists a homomorphism $\nu$ in $\mathcal{F}_{C_{n}}$ such that, for any $\varphi\in\Gamma_{0}=\{\varphi\in \bF(\Sigma, \mathcal{V}) : \exists\textsf{L}\in \mathbb{B}_{n}(\textsf{L}(\varphi)\in \Gamma)\}$, if $\textsf{L}(\varphi)\in \Gamma$ then $\nu(\varphi)=L$. As before, we start with a (well-defined) function $\nu_{0}:\Gamma_{0}\rightarrow B_{n}$ such that $\nu_{0}(\varphi)=L$ iff $\textsf{L}(\varphi)\in\Gamma$. This function must be extended to a $\nu \in \mathcal{F}_{C_{n}}$ satisfying  $\Gamma$. For the clauses for Hintikka sets starting with $\laL(\varphi)\in\Gamma$  in which the immediate subformulas of $\varphi$ are also in $\Gamma_{0}$, the proof runs straightforwardly. As in the case of $\bbT_{1}$, difficulties only appear in the other clauses, namely clauses $9$, $10$, $11$, $13$ and $14$.

This is not actually a problem: take clause $11$ to scrutinize. Starting with $\lat^{n}_{i}(\varphi\vee\psi) \in\Gamma$ assume, without loss of generality, that $\lat^{n}_{k}(\varphi)\in\Gamma$, but $\lat^{n}_{j}(\psi)\notin\Gamma$ for all $0\leq j\leq n-1$ (the symmetric case is proved analogously). By construction,  $\nu(\varphi\vee\psi)=t^{n}_{i}$ and $\nu(\varphi)=t^{n}_{k}$, while $\nu(\psi)$ can be arbitrarily defined, only respecting the clauses for $\nu$ being in  $\mathcal{F}_{C_{n}}$. This could, in principle, lead to an incongruence between $\nu(\varphi)$, $\nu(\psi)$ and $\nu(\varphi\vee \psi)$. However, this is guaranteed to not happen, as the definition of a Hintikka set carries sufficient conditions of coherence (in this case for $\lat^{n}_{i}(\varphi\vee\psi)$ and $\lat^{n}_{k}(\varphi)$). Looking at the table for $\tilde{\vee}$ for $\mathcal{A}_{C_{n}}$, we see that regardless of the value taken by $\nu(\psi)$, one has $\nu(\varphi)\tilde{\vee}\nu(\psi)=t^{n}_{j}\tilde{\vee}\nu(\psi)=D_{n}$, which certainly contains $\nu(\varphi\vee\psi)=t^{n}_{i}$ and then $\nu$ will be an homomorphism, for any choice of the value taken by $\nu(\psi)$. Similar situations occur when the other clauses  $9$, $10$, $13$ and $14$ are analyzed.

Once we have that every Hintikka set for $\bbT_{n}$ has a corresponding homomorphism in $\mathcal{F}_{C_{n}}$ that satisfies it, it remains to prove that the set of labelled formulas occurring in an open branch $\theta$ of a completed tableau $\mathcal{T}$ in $\bbT_{n}$ is a Hintikka set for $\bbT_{n}$. But this is obvious from the definitions. All of these results allow one to prove the completeness of tableau systems  w.r.t.  $\matGM_{C_n}$. This means that we have a decision method for $C_{n}$ in $\bbT_{n}$, because of the characteristics of  $\bbT_{n}$.

Although $\bbT_{n}$ is a decision procedure for $C_n$, it is clear that the width of the generated trees grows rapidly, due to the large number of branches generated by the tableau rules. In the next section we shall see that the size of tableaux in $\bbT_{n}$ can be drastically reduced by considering derived rules.

\subsection{Derived tableau rules} \label{der-rules}

In this section we introduce some derived rules for  $\bbT_{1}$ and  $\bbT_{n}$, for $n \geq 2$. The use of such rules allows us to reduce the size of the generated tableaux. The proof of soundness of such rules is almost immediate. A symbol $\star$ means that the branch immediately closes after applying such rule. For  $\bbT_{1}$ we define the following rules:
$$\begin{array}{cccc}
\displaystyle \frac{\laT(\varphi \land \neg\varphi)}{\lat(\varphi)} 
& \displaystyle \frac{\laF(\varphi \land \neg\varphi)}{\laT(\varphi) \mid \laF(\varphi)} &  \displaystyle \frac{\laT(\varphi^{\circ})}{\laT(\varphi) \mid \laF(\varphi)} &  \displaystyle \frac{\lat(\varphi^{\circ})}{\star} 
\end{array}$$
$$\begin{array}{cccc}
\displaystyle \frac{\laF(\varphi^{\circ})}{\lat(\varphi)} 
&  \displaystyle \frac{\laT(\varphi^{\circ} \land \psi^{\circ})}{\begin{array}{c|c|c|c}\laT(\varphi) & \laT(\varphi) & \laF(\varphi) & \laF(\varphi) \\ \laT(\psi) & \laF(\psi) & \laT(\psi) & \laF(\psi)\end{array}} & \displaystyle \frac{\lat(\varphi^{\circ} \land \psi^{\circ})}{\star} & \displaystyle \frac{\laF(\varphi^{\circ} \land \psi^{\circ})}{\lat(\varphi) \mid \lat(\varphi)} \\[2mm]
&&\\[2mm]
\end{array}
$$
For  $\bbT_{n}$, with $n\geq 2$, consider the set of labels $\textsf{D}_n^{\leq i} = \{\laT_{n}, \lat^{n}_{0}, \ldots , \lat^{n}_{i}\}$, $\textsf{D}_n^{\geq i} =\{\lat^{n}_{i}, \ldots , \lat^{n}_{n-1}\}$, for $0 \leq i \leq n-1$, and $\textsf{D}_n^{\leq -1} = \{\laT_{n}\}$. Then, we can consider the following derived rules for  $\bbT_{n}$, where $0 \leq i \leq n-1$; $r \geq n$; $0 \leq k \leq n-2$; $p \geq n-1$; $1 \leq j \leq n$;  $1 \leq u \leq n-1$; $0 \leq s \leq n - u - 1$; and $n-u \leq m \leq n-1$:
$$\begin{array}{cccc}
\displaystyle \frac{\laT_{n}(\varphi^i \land \neg\varphi^i)}{\lat^{n}_{i}(\varphi)} &  \displaystyle \frac{\laT_{n}(\varphi^r \land \neg\varphi^r)}{\star}  &   \displaystyle \frac{\lat^{n}_{i}(\varphi^k \land \neg\varphi^k)}{\textsf{D}_n^{\geq k+1}(\varphi)} &  \displaystyle \frac{\lat^{n}_{i}(\varphi^p \land \neg\varphi^p)}{\star} \\[2mm]
&&\\[2mm]
\displaystyle \frac{\laF_{n}(\varphi^i \land \neg\varphi^i)}{\textsf{D}_n^{\leq i-1}(\varphi) \mid \laF_{n}(\varphi)} & \displaystyle \frac{\laT_{n}(\varphi^j)}{\textsf{D}_n^{\leq j-2}(\varphi) \mid \laF_{n}(\varphi)} & \displaystyle \frac{\lat^{n}_{s}(\varphi^u)}{\lat^{n}_{s+u}(\varphi)} 
&  \displaystyle \frac{\lat^{n}_{m}(\varphi^u)}{\star}\\[2mm]
&&\\[2mm]
\displaystyle \frac{\laF^{n}(\varphi^j)}{\lat^{n}_{j-1}(\varphi)} &  \displaystyle \frac{\laT^{n}(\neg\varphi^j)}{\textsf{D}_n^{\geq j-1}(\varphi)} &  \displaystyle \frac{\lat^{n}_{i}(\neg\varphi^u)}{\textsf{D}_n^{\geq u}(\varphi)} &  \displaystyle \frac{\laF_{n}(\neg\varphi^j)}{\textsf{D}_n^{\leq j-2}(\varphi) \mid \laF_{n}(\varphi)} \\[2mm]
&&\\[2mm]
\end{array}$$

Finally, any branch containing either $\lat^{n}_{i}(\varphi^r)$ or  $\lat^{n}_{i}(\neg\varphi^r)$ must close, for every $0 \leq i \leq n-1$ and $r\geq n$. The derived rules for  $\bbT_{n}$ faithfully reflect the information contained in    Table~\ref{tabela3}. Of course other derived rules could be considered, helping to obtain shorter derivations in  $\bbT_{n}$, for $n \geq 1$.


\section{Final remarks} \label{FinRem}

This paper introduces a new semantics for da Costa's calculi $C_n$, which constitutes a relatively simple decision procedure for these logics inducing, in addition, a second decision procedure by means of tableau systems. The semantics is based on the notion of restricted non-deterministic matrix  (in short, RNmatrix) semantics. The great advantage of RNmatrices with respect to non-deterministic matrices is that they allow to obtain, in certain cases, the characterization,  by means of a single finite RNmatrix, of a logic which is not characterizable by a single finite Nmatrix.  

Different from what happens with the examples presented here, the characterization of a logic by means of a finite RNmatrix do not always ensure the existence of an effective decision procedure. For instance, the four-valued characteristic RNmatrix $\mathcal{K}_{\bf L}=(\mathcal{A}_{\bf L},\{T\},\mathcal{F}_{\bf L})$ for ${\bf L} \in \{{\bf T},  {\bf S4}, {\bf S5}\}$ introduced by Kearns (recall Example~\ref{Kearns}) is far from defining a decent  decision procedure for the modal logic {\bf L}. The drawback  is that the decision problem of the set $\mathcal{F}_{\bf L}$ of valuations is equivalent to the decision problem for {\bf L} itself. This means that, as it stands, the RNmatrix $\mathcal{K}_{\bf S4}$, for instance, does not contribute for the decision problem of modal logic {\bf S4}. 
For an additional discussion on this topic see~\cite[Section~4]{con:far:per:16}.

Even if the set of valuations of a finite RNmatrix is decidable, dealing with it can be a very complex task.
As mentioned in Subsection~\ref{FCn}, the set of valuations  $\mathcal{F}_{\bf C}^n$ of the  RNmatrix $\mathcal{M}_{\bf C}^n$ associated to the characteristic Fidel structure {\bf C} for $C_n$  is very complicated. In this sense, the RNmatrices proposed here for $C_n$ constitute an advance with respect to the decision problem of $C_n$. In particular, RNmatrices (and tableau systems) for \cila, \mbccl\ and $C_1$ simplify the decision procedures obtained in~\cite{Avron:Arieli:Zamansky:18} from infinite Nmatrices.  

The RNmatrix semantics presented here allows us to conceive da Costa's hierarchy of $C$-systems as  a family of (non deterministically) $(n+2)$-valued logics, where $n$ is the number of ``inconsistently true'' truth-values and 2 is the number of ``classical'' or ``consistent'' truth-values (the truth and the false), for every calculus $C_n$. We consider that this novel and elucidative interpretation deserves to be analyzed from a philosophical perspective.

In general, the class of Fidel structures of a given logic has interesting formal properties. In~\cite{CF:20} was proposed the study of Fidel structures as being first-order (Tarskian) structures modeling certain (Horn) axioms, allowing so to study them within the rich framework of model theory. From the observation we made at Subsection~\ref{FCn}, Fidel structures could be alternatively analyzed from the perspective of category theory, specifically within the category of multialgebras. This is  a topic that deserves future research.

The concrete examples of finite-valued characteristic RNmatrices  presented here show that restricted non-deterministic matrices constitute a powerful and promising semantical framework for non-classical logics.
\paragraph{Acknowledgements.} 

The first author acknowledges support from  the  National Council for Scientific and Technological Development (CNPq), Brazil
under research grant 306530/2019-8. The second author was supported by a doctoral scholarship from CAPES, Brazil. 


\end{document}